\documentclass[notitlepage,11pt, reqno]{amsart}
\usepackage{amsmath,amssymb,nicefrac,bm,upgreek,verbatim,mathtools}
\usepackage{dsfont}     
\usepackage[mathscr]{eucal}
\usepackage{color,verbatim}
\usepackage{url}
\usepackage[normalem]{ulem}

\definecolor{dmagenta}{rgb}{.4,.1,.5}
\definecolor{dblue}{rgb}{.0,.0,.6}
\definecolor{ddblue}{rgb}{.0,.0,.4}
\definecolor{dred}{rgb}{.5,.0,.0}
\definecolor{dgreen}{rgb}{.0,.4,.0}
\definecolor{Eeom}{rgb}{.0,.0,.5}

\numberwithin{equation}{section}

\newtheorem{theorem}{Theorem}[section]
\newtheorem{lemma}{Lemma}[section]

\theoremstyle{definition}

\newtheorem{example}{Example}[section]
\newtheorem{condition}{Condition}[section]

\theoremstyle{remark}

\newcommand{\df}{:=}
\DeclareMathOperator{\Exp}{\mathbb{E}}
\DeclareMathOperator{\Prob}{\mathbb{P}}
\newcommand{\D}{\mathrm{d}}   
\newcommand{\E}{\mathrm{e}}   
\newcommand{\RR}{\mathbb{R}}  
\newcommand{\Rd}{\mathbb{R}^d}
\newcommand{\NN}{\mathbb{N}}   

\newcommand{\Ind}{\mathds{1}}        
\newcommand{\Act}{\mathbb{U}}        
\newcommand{\pAct}{\tilde{\mathbb{U}}}
\newcommand{\Uadm}{\mathfrak{U}}     
\newcommand{\Usm}{\mathfrak{U}^{\mathrm{SM}}}  
\newcommand{\sJmin}{\sJ_{\mathrm{inf}}}
\newcommand{\sJmax}{\sJ_{\mathrm{sup}}}
\newcommand{\tu}{\tilde{u}}
\newcommand{\lamstr}{\lambda^{*}}
\newcommand{\lammax}{\lambda_{\mathrm{max}}}

\newcommand{\Sob}{\mathscr{W}} 
\newcommand{\Sobl}{\mathscr{W}_{\mathrm{loc}}}  
\newcommand{\transp}{^{\mathsf{T}}}  

\newcommand{\Lg}{\mathscr{L}}                

\newcommand{\uuptau}{\Breve{\uptau}}
\newcommand{\grad}{\nabla}

\newcommand{\sB}{\mathscr{B}}     
\newcommand{\Cc}{\mathcal{C}}     
\newcommand{\sC}{\mathscr{C}}
\newcommand{\sE}{\mathscr{E}}    
\newcommand{\sJ}{\mathscr{J}}
\newcommand{\sK}{\mathscr{K}}     
\newcommand{\Lp}{L}   

\newcommand{\Lyap}{\mathcal{V}}   

\newcommand{\abs}[1]{\lvert#1\rvert}
\newcommand{\norm}[1]{\lVert#1\rVert}

\usepackage{accents}
\newlength{\dhatheight}

\begin{document}

\title
{Zero-sum stochastic differential game with risk-sensitive cost}

\author{Anup Biswas}
\address{Department of Mathematics,
Indian Institute of Science Education and Research,
Dr. Homi Bhabha Road, Pune 411008, India.}
\email{anup@iiserpune.ac.in}
\author{Subhamay Saha}
\address{Department of Mathematics,
Indian Institute of Technology, Guwahati 781039, India.}
\email{saha.subhamay@iitg.ernet.in}

\date{\today}

\begin{abstract}
Zero-sum games with risk-sensitive cost criterion are considered with underlying dynamics being given by controlled stochastic
differential equations. Under the assumption of geometric stability on the dynamics, we completely
characterize all possible saddle point strategies in the class of stationary Markov controls. In addition, we also establish
existence-uniqueness result for the value function of the Hamilton-Jacobi-Isaacs equation.

\end{abstract}


\maketitle

\section{\bf Introduction}
This article concerns with zero-sum game for risk-sensitive costs where the underlying dynamics of the players are
given by non-degenerate controlled stochastic differential equations. There are two players, one of them
tries to minimize his/her payoff and the other one tries to maximize it.
The maximizing player may be thought
of as {\it nature}
who may choose a parametrization of the model to be totally adverse to the other controller, in a {\it non-anticipative} way.
 The admissible controls (or actions) of each 
player are assumed to take  values in a space of probability measures over a compact metric space. In other words,
each player chooses his/her control from a given collection of mixed strategies. 
As the system evolves, the costs
are accumulated and the performance of a control policy is measured by a long-run average of risk-sensitive cost 
criterion. The main objective of this work is to find solution for the associated Hamilton-Jacobi-Isaacs (HJI) equation
characterizing the value of the game and rendering saddle point strategies. The key results of this article can be
roughly described as follows.
\begin{itemize}
\item {\it Existence and Uniqueness of solution.}\, The HJI equation has an eigenpair $(V, \lambda)\in\Cc^2(\Rd)\times\RR$
where $\lambda$ is the value of the game. Moreover, this solution is unique.

\item{\it Verification result.}\, Any measurable mini-max  selector of the HJI equation form a saddle point strategy, and any
saddle point strategy in the class of stationary Markov controls can be obtained as a measurable mini-max selector of
the HJI equation.
\end{itemize}
Risk-sensitive (or {\it exponential-of-integral}) control has been an active area of research in the field of control theory. 
These problems were introduced by Jacobson \cite{Jacobson} in seventies. We also refer to Whittle \cite{Whittle}
and references therein for an early account of the risk-sensitive control literature. Most of these problems have 
concentrated on minimizing the risk-senstive payoff.  Among the most relevant to the present are \cite{ari-anup, ABS, 
biswas-11a, biswas-11, Biswas-10, Fleming-95, Nagai-96}. On the other hand the number of literature on
stochastic differential games with risk-sensitive cost criterion are very few \cite{Basar, Basu-Ghosh, El-Ham, Ghosh-16}.
Ba\c{s}ar \cite{Basar} considers non-zero sum game with finite horizon risk-sensitive costs and establishes existence of Nash equilibrium
whereas Ghosh et.al. \cite{Ghosh-16} obtain existence of Nash equilibrium with  long-run average of risk-sensitive payoff. 
Both zero-sum and 
nonzero-sum games with finite horizon risk-sensitive cost criterion have been 
considered by El-Karoui and Hamad\`{e}ne \cite{El-Ham} where the authors have
used backward stochastic differential equation to
prove existence of a saddle-point and an equilibrium point for respectively, the zero-sum and nonzero-sum games.
In a recent work Basu and Ghosh \cite{Basu-Ghosh} consider a zero-sum game with long--run average of risk-sensitive cost
and adopt the standard vanishing discount approximation
to establish existence of saddle points which are mini-max selector of the associated HJI equation.
This work is done with bounded drift, diffusion coefficient and cost functions. After a careful reading we have found that the proofs
of \cite{Basu-Ghosh}(see Theorem~3 and ~4 there) contain some crucial error that do not seem to have an easy fix. 
In \cite{Basu-Ghosh} the authors have constructed an eigenpair for the HJI equation and have made an attempt to show
 that this eigenvalue is the value of the game. But this can not be true as a recent work of Berestycki and Rossi
 \cite{Berestycki-15} suggests that
 there are uncountably many eigenvalues (for the uncontrolled linear operator) with positive eigenfunctions. Therefore, one has to look for the {\it principal} eigenvalue of HJI equation. In particular, one has to show that the eigenvalue obtained by the vanishing discount method is actually the principal eigenvalue
 for the HJI equation.
 We show in this article that under certain conditions, the value of the game can be seen as the principal eigenvalue
of the HJI equation. This problem is closely related to the charatecterization of the principal eigenvalue (or \textit{optimal value}) for the Hamilton-Jacobi-Bellman(HJB) type equations
which have been studied by several authors recently \cite{ABS, KS06, Ich13}. Let us also mention \cite{BM16, C14, KN}
where eigenvalue type problems for the HJB equation
related to the the ergodic control problems in $\Rd$ are studied.

In this article we consider a large class (compared to \cite{Basu-Ghosh}) of control problems. We prove not only
the existence of a saddle point strategy, but also characterize all possible saddle point strategies
in the class of stationary Markov controls. In addition, we also obtain the uniqueness of the value function for the HJI equations.
To find the {\it principal} eigenpair of the HJI equation we start with the Dirichlet eigenvalue value problems in
bounded domains where existence of principal
eigenpair is known and then increase these domains to $\Rd$. This idea was used earlier by Biswas in \cite{biswas-11a}
but for a simpler setting.
The approach of \cite{Basu-Ghosh, Ghosh-16}
is based on parabolic PDE where one introduces an additional {\it risk-sensitive parameter} $\theta$ multiplied with the running
cost and studies the associated parabolic PDE, viewing $\theta$ as a time parameter. It seems this approach is not very much helpful in studying our game problem whereas
 the eigenvalue approach of \cite{biswas-11a} (see also \cite{C14})
 seems more natural in current situation. Another advantage in using the latter approach is the ease 
 in obtaining
 stochastic representation of the value functions which helps us to apply strong maximum principle.
 It should be kept in mind that \cite{biswas-11a} deals with a minimization problem and it does not discuss characterization 
of the optimal stationary Markov controls. It turns out that the game problem is much more
involved than the minimization problem. One of the reasons for this difficulty
 is that our game is a generalization to both minimization and maximization problems with risk-sensitive cost, and it
 is not easy to convert a maximization problem to a minimization problem of the same type. Therefore, we
 separately study the maximization problem with risk-sensitive costs in Section~\ref{S-max}  (see Theorem~\ref{T3.1}) and characterize
 the optimal stationary Markov controls. 
 
 We borrow several results from \cite{ABS} where an interesting characterization 
 of the strict monotonicity property of the principal eigenvalues of second order elliptic operators has been studied. Using
 this characterization the authors then study a minimization problem in \cite{ABS} for the  risk-sensitive cost. We also use the
 stochastic representation formula  of the principal eigenfunction from \cite{ABS}.
 The problem we consider in this article does not follow in a straightforward manner from \cite{ABS}
  and the level of technical difficulties is also more involved.
 For instance, in case of the minimization problem it is seen in \cite{ABS, biswas-11a} that any limit of the principal eigenvalues of the HJB operator considered in
 bounded domains, as the domains increase to $\Rd$, is suboptimal, and therefore, by choosing a selector of the HJB one proves that the
 limiting eigenvalue is actually the principal one and hence, it has to be optimal. But this simple policy does not seem to work for the
 maximization problem. So to study the maximization problem we first perturb the cost to a \text{norm-like} running cost and show that
 the maximization problem can be solved for this perturbed cost. Then using some stability estimate we could show that it is possible to pass to
 the limit and solve the original maximization problem (see Lemma~\ref{L3.4} and ~\ref{L3.5} below).

Rest of the article is organized as follows. In the remaining part of this section we summarize the notations. Section 2 introduces the problem setup and the main result. In section 3, we study a maximization problem with risk-sensitive criterion, and finally section 4 deals with the proof of the main Theorem.

\subsection*{Notation.}
The standard Euclidean norm in $\RR^{d}$ is denoted by $\abs{\,\cdot\,}$.
The set of nonnegative real numbers is denoted by $\RR_{+}$,
$\NN$ stands for the set of natural numbers, and $\Ind$ denotes
the indicator function.
Given two real numbers $a$ and $b$, the minimum (maximum) is denoted by $a\wedge b$ 
($a\vee b$), respectively.
The closure, boundary, and the complement
of a set $A\subset\Rd$ are denoted
by $\Bar{A}$, $\partial{A}$, and $A^{c}$, respectively.
We denote by $\uptau(A)$ the \emph{first exit time} of the process
$\{X_{t}\}$ from the set $A\subset\RR^{d}$, defined by
\begin{equation*}
\uptau(A) \;\df\; \inf\;\{t>0\;\colon\, X_{t}\not\in A\}\,.
\end{equation*}
The open ball of radius $r$ in $\RR^{d}$, centered at the origin,
is denoted by $B_{r}$, and we let $\uptau_{r}\df \uptau(B_{r})$,
and $\uuptau_{r}\df \uptau(B^{c}_{r})$.

The term \emph{domain} in $\RR^{d}$
refers to a nonempty, connected open subset of the Euclidean space $\RR^{d}$. 
For a domain $D\subset\RR^{d}$,
the space $\Cc^{k}(D)$ ($\Cc^{\infty}(D)$)
refers to the class of all real-valued functions on $D$ whose partial
derivatives up to order $k$ (of any order) exist and are continuous,
and $\Cc_{b}(D)$ denotes the set of all bounded continuous
real-valued functions on $D$.
In addition $\Cc_c(D)$ denotes the class of functions in $\Cc(D)$ that
have compact support, and $\Cc_0(\Rd)$ the class of continuous
functions on $\Rd$ that vanish at infinity. 
By a slight abuse of notation,
whenever the whole space $\RR^d$ is concerned, we write
$f\in\Cc^{k}(\RR^{d})$ 
whenever $f\in\Cc^{k}(D)$ for all bounded domains $D\subset\RR^{d}$.
The space $\Lp^{p}(D)$, $p\in[1,\infty)$, stands for the Banach space
of (equivalence classes of) measurable functions $f$ satisfying
$\int_{D} \abs{f(x)}^{p}\,\D{x}<\infty$, and $\Lp^{\infty}(D)$ is the
Banach space of functions that are essentially bounded in $D$. We shall use the notation
$\norm{\cdot}_\infty$ to denote the $L^\infty$ norm on the underlying domain.
The standard Sobolev space of functions on $D$ whose generalized
derivatives up to order $k$ are in $\Lp^{p}(D)$, equipped with its natural
norm, is denoted by $\Sob^{k,p}(D)$, $k\ge0$, $p\ge1$.

In general, if $\mathcal{X}$ is a space of real-valued functions on $D$,
$\mathcal{X}_{\mathrm{loc}}$ consists of all functions $f$ such that
$f\varphi\in\mathcal{X}$ for every $\varphi\in\Cc_{c}^{\infty}(D)$,
the space of smooth functions on $D$ with compact support.
In this manner we obtain for example the space $\Sobl^{2,p}(D)$.


We adopt the notation
$\partial_{i}\df\tfrac{\partial~}{\partial{x}_{i}}$ and
$\partial_{ij}\df\tfrac{\partial^{2}~}{\partial{x}_{i}\partial{x}_{j}}$
for $i,j\in\NN$.
We often use the standard summation rule that
repeated subscripts and superscripts are summed from $1$ through $d$.
For example,
\begin{equation*}
\tfrac{1}{2}\, a^{ij}\partial_{ij}\varphi
+ b^{i} \partial_{i}\varphi \;\df\; \tfrac{1}{2}\, \sum_{i,j=1}^{d}a^{ij}
\frac{\partial^{2}\varphi}{\partial{x}_{i}\partial{x}_{j}}
+\sum_{i=1}^{d} b^{i} \frac{\partial\varphi}{\partial{x}_{i}}\,.
\end{equation*}

\section{\bf Settings and main results}
The following assumptions (A1)-(A3) on the controlled diffusion \eqref{E2.1}
will be in effect throughout this article unless
otherwise mentioned.
\begin{itemize}
\item[(A1)]
\emph{Local Lipschitz continuity:\/}
The function
\begin{equation*}
b\;=\;\bigl[b^{1},\dotsc,b^{d}\bigr]\transp\,\colon\,\RR^{d}\times\Act_1\times\Act_2\to\RR^{d},\quad \upsigma\;=\;\bigl[\upsigma^{ij}\bigr]\,\colon\,\RR^{d}\to\RR^{d\times d}
\end{equation*}
are locally Lipschitz in $x$ (uniformly in the other two variables for $b$) with a Lipschitz constant $C_{R}>0$
depending on $R>0$.
In other words, we have
\begin{equation*}
\abs{b(x, u)-b(y, u)} + \norm{\upsigma(x) - \upsigma(y)}
\;\le\;C_{R}\,\abs{x-y}\qquad\forall\,x,y\in B_R\,, \; u\in\Act_1\times\Act_2\, .
\end{equation*}
We also assume that the action spaces $\Act_i, i=1,2,$ are compact metric spaces and 
$b$ is jointly continuous in $(x, u)$.
\smallskip
\item[(A2)]
\emph{Affine growth condition:\/}
$b$ and $\upsigma$ satisfy a global growth condition of the form
\begin{equation*}
\sup_{u\in\Act_1\times\Act_2}\, \langle b(x, u),x\rangle^{+} + \norm{\upsigma(x)}^{2}\;\le\;C_0
\bigl(1 + \abs{x}^{2}\bigr) \qquad \forall\, x\in\RR^{d},
\end{equation*}
for some constant $C_0>0$,
where $\norm{\upsigma}^{2}\;\df\;
\mathrm{trace}\left(\upsigma\upsigma\transp\right)$.
\smallskip
\item[(A3)]
\emph{Nondegeneracy:\/}
For each $R>0$, it holds that
\begin{equation*}
\sum_{i,j=1}^{d} a^{ij}(x)\xi_{i}\xi_{j}
\;\ge\;C^{-1}_{R} \abs{\xi}^{2} \qquad\forall\, x\in B_{R}\,,
\end{equation*}
and for all $\xi=(\xi_{1},\dotsc,\xi_{d})\transp\in\RR^{d}$,
where $a\df \upsigma \upsigma\transp$.
\end{itemize}
Set of all probability measures on $\Act_i$ is denoted by $\pAct_i$, $i=1,2$. 
In our control model we have
two players and $\pAct_i$ denotes the {\it relaxed} action space for the the $i$-th player, $i=1,2$. 
We extend the drift $b:\Rd\times\pAct_1\times\pAct_2\to\Rd$ as follows: 
for $\nu_i\in\pAct_i,\, i=1,2,$
$$b(x, \nu_1, \nu_2)=\int_{\Act_1\times\Act_2} b(x, u_1, u_2)\, \nu_1(\D{u_1})\, \nu_2(\D{u_2}).$$
It is easy to verify that the extended drift satisfies (A1)-(A2) with $\Act_i$ replaced by $\pAct_i, \, i=1,2$.
One major advantage of this extension is that $b(x, \cdot, \cdot):\pAct_1\times\pAct_2\to\Rd$ becomes 
coordinate-wise  convex. This extension would play a key role in the selection of saddle point below.

The controlled stochastic differential equation (SDE) is given by
\begin{equation}\label{E2.1}
\D{X_s} \;=\; b(X_s, U^1_s, U^2_s)\, \D{s} + \upsigma(X_s)\, \D{W_s},
\end{equation}
where $W$ is a standard $d$-dimensional Wiener process on some complete, filtered probability space $(\Omega, \mathfrak{F}, \{\mathfrak{F}_t\}, \Prob)$, and $U^i$ is an $\pAct_i$ valued process satisfying following admissibility 
condition: for $s<t$, the completion of $\sigma\{U^i_r, W_r\; :\; r\leq s, i=1,2\}$ relative to $\{\mathfrak{F}, \Prob\}$ is independent of $W_t-W_s$. Let us clarify that we do not fix any probability space a priori and this is an important
technical point which allow as to consider a fairly large class of {\it admissible} control.
It is well known that given a complete, filtered probability space $(\Omega, \mathfrak{F}, \{\mathfrak{F}_t\}, \Prob)$
with a Wiener process $W$, 
under (A1)--(A3), for any progressively measurable $(U^1, U^2)$ 
there exists a unique solution of \eqref{E2.1}
\cite[Theorem~2.2.4]{book}.
We define the family of operators $\Lg^{\tilde{u}}:\Cc^{2}(\RR^{d})\mapsto\Cc(\RR^{d})$,
where $\tilde{u}\in\pAct_1\times\pAct_2$ plays the role of a parameter, by
\begin{equation}\label{E-Lg}
\Lg^{\tilde{u}} f(x) \;=\; \tfrac{1}{2}\, a^{ij}(x)\,\partial_{ij} f(x)
+ b^{i}(x,\tu_1, \tu_2)\, \partial_{i} f(x)\,,\quad \tu=(\tu_1, \tu_2)\in\pAct_1\times\pAct_2\,.
\end{equation}
By a stationary Markov control for the $i$-th player we mean a control of the form $U^i_t=v_i(X_t)$ for a Borel measurable map $v_i:\Rd\to\pAct_i, \, i=1,2$. By an abuse of notation we will refer to $v_i$ as stationary Markov control.
Let $\Usm_{i}$ denote the set of all stationary Markov controls for the $i$-th player.
It is well known that under $v_i\in\Usm_i$
\eqref{E2.1} has a unique strong solution \cite{Gyongy-96}.
Moreover, under $v=(v_1, v_2)\in\Usm_1\times\Usm_2$, the process $X$ is strong Markov,
and we denote its transition kernel by $P^{t}_{v}(x,\cdot\,)$.
It also follows from the work in \cite{Bogachev-01} that under
$v\in\Usm_1\times\Usm_2$, the transition probabilities of $X$
have densities which are locally H\"older continuous.
Thus $\Lg_{v}$ defined by
\begin{equation*}
\Lg_{v} f(x) \;=\; \tfrac{1}{2}\, a^{ij}(x)\,\partial_{ij} f(x)
+ b^{i} \bigl(x,v_1(x), v_2(x)\bigr)\, \partial_{i} f(x)\,,\quad v=(v_1, v_2)\in\Usm_1\times\Usm_2\,,
\end{equation*}
for $f\in\Cc^{2}(\RR^{d})$,
is the generator of a strongly-continuous
semigroup on $\Cc_{b}(\RR^{d})$, which is strong Feller.
When $v\in\Usm_1\times\Usm_2$ we use $v$ as subscript in $\Lg_v$ to distinguish
it from $\Lg^{\tilde{u}}$, $\tilde{u}\in\pAct_1\times\pAct_2$, defined in the preceding paragraph.
We let $\Prob_{x}^{v}$ denote the probability measure and
$\Exp_{x}^{v}$ the expectation operator on the canonical space of the
process under the control $v\in\Usm_1\times\Usm_2$, conditioned on the
process $X$ starting from $x\in\RR^{d}$ at $t=0$. For every $i$ the set $\Usm_i$ is metrizable with a compact metric
\cite[Section~2.2.4]{book}, \cite{Borkar-89}. In fact, $v_n\to v$ in $\Usm_i$ if and only if
$$\int_{\Rd} f(x) \int_{\Act_i} g(x, u_i) v_n(\D{u_i}|x)\, \D{x}\xrightarrow[]{n\to\infty}\int_{\Rd} f(x) \int_{\Act_i} g(x, u_i) v(\D{u_i}|x)\, \D{x},$$
for all $f\in L^1(\Rd)\cap L^2(\Rd)$ and $g\in\Cc_b(\Rd\times\Act_i)$.
Recall that $\uptau(D)$ denotes the first exit time of the process $X$ from domain $D$. A pair stationary Markov controls $(v_1, v_2)\in\Usm_1\times\Usm_2$
is said to be stable if the associated process is positive recurrent i.e.,
$\Exp^v_x[\uptau(D^c)] < \infty$ for all $x \in D^c$. It is known that for a non-degenerate diffusion the property of  positive recurrence is independent of domain, i.e., if it holds for one domain $D$ then it also holds for every domain \cite[Theorem~2.6.10]{book}.

Let us now introduce the set of all admissible controls for our game problem.
We follow the approach of \cite{Bor-Ghosh, Mannucci}. A control
is called a \textit{feedback control} if it is progressively measurable with respect to the natural filtration generated by $X$.
More precisely, we say $U^i$ is of feedback form if for some measurable $f_i: [0, \infty)\times\Cc([0, \infty):\Rd)\to\pAct_i$
we have $U^i_t=f_i(t, X_{[0, t]}), i=1,2$. It should be noted that stationary Markov controls are also feedback controls.
The set of all feedback controls for the $i$-th player is denoted by $\Uadm_i, i=1,2$. We also refer the members of $\Uadm_i$
as admissible controls. By \cite[Theorem~2.2.11]{book} we know that for any $(U^1, U^2)\in\Uadm_1\times\Uadm_2$ \eqref{E2.1} has a unique weak solution. 

A Borel measurable function $\ell:\Rd\to\RR$ is said to be inf-compact if 
for any $\kappa\in\RR$ the set $\{x\in\Rd\; \colon \ell(x)\leq\kappa\}$ is compact. It is clear that for any inf-compact function $\ell$
 we have $\lim_{\abs{x}\to\infty} \ell(x)=\infty$.  For 
a locally bounded function $f:\Rd\to\RR$ and $(U^1, U^2)\in\Uadm_1\times\Uadm_2$, the risk-sensitive average is defined to be
$$\sE_x(f, U^1, U^2)=\limsup_{T\to\infty}\, \frac{1}{T}\, \log \Exp_x\left[\E^{\int_0^T f(X_s, U^1_s, U^2_s)\, \D{s}}\right], \; \text{and}\; \sE(f, U^1, U^2)=\inf_{x\in\Rd} \sE_x(f, U^1, U^2).$$
The running cost function for our problem is given by
a continuous function $c:\Rd\times\Act_1\times\Act_2\to\RR_+$ which is locally Lipschitz continuous in its first argument uniformly with respect to $u\in\Act_1\times\Act_2$. Without loss of generality we assume that $c$ is non-constant function.
As earlier we extend $c$ over $\Rd\times\pAct_1\times\pAct_2$ as follows: for $\nu_i\in\pAct_i, \, i=1,2$
$$c(x, \nu_1, \nu_2)=\int_{\Act_1\times\Act_2} c(x, u_1, u_2)\, \nu_1(\D{u_1})\, \nu_2(\D{u_2}).$$
In this article the cost criterion is given by $\sE_x(c, U^1, U^2)$, and for notational economy we shall drop the notation $c$ and denote it by $\sE_x(U^1, U^2)$. We also define
\begin{equation}\label{E2.3}
\sJmin(x, U^2)=\inf_{U^1\in\Uadm_1}\sE_x(U^1, U^2), \quad \sJmax(x, U^1)=\sup_{U^2\in\Uadm_2}\sE_x(U^1, U^2).
\end{equation}
Therefore the upper and lower value of the game are respectively defined as
\begin{align*}
\bar{\Lambda} &= \inf_{U^1\in\Uadm_1}\sJmax(x, U^1)= \inf_{U^1\in\Uadm_1}\sup_{U^2\in\Uadm_2}\sE_x(U^1, U^2),
\\
\underline{\Lambda} &=\sup_{U^2\in\Uadm_2}\sJmin(x, U^2)= \sup_{U^2\in\Uadm_2}\inf_{U^1\in\Uadm_1}\sE_x(U^1, U^2).
\end{align*}
The game is said to have a value if we have
$$\bar\Lambda=\underline{\Lambda}=\Lambda\; \text{(say)}.$$
The reader might have observed that we do not write $\Lambda$ (or $\bar\Lambda$, $\underline{\Lambda}$) as a function of $x$.
We show that under certain stability hypothesis (assumed below), the value of the game is independent of $x$.
Let us now introduce two set of stability assumptions (Conditions~\ref{C2.1} and ~\ref{C2.2}) that will be used in this article. These conditions
are generally referred to as the geometric stability conditions and have been heavily used in the study of
 discrete and continuous time Markov control problems.
In the context of risk-sensitive controls similar conditions have been used by \cite{ABS, Basu-Ghosh, biswas-11, Fleming-95}.

\begin{condition}\label{C2.1}
There exists a positive function $\Lyap\in\Cc^2(\Rd), \, \inf_{\Rd}\Lyap>0,$ and a constant $\gamma>0$ such that
\begin{equation}\label{lyapunov}
\max_{u\in\Act_1\times\Act_2} \Lg^u \Lyap \;\leq \; \beta\, \Ind_{\sK} - \gamma \Lyap,
\end{equation}
for some compact $\sK\subset\Rd$ and constant $\beta$. Moreover, the cost function $c$ is bounded, and
\begin{equation}\label{E2.5}
 \norm{c}_\infty\;<\;{\gamma}.
\end{equation}
\end{condition}

\begin{example}\label{Eg1}
Suppose $a=\mathrm{Id}$ and 
$$b(x, u)\cdot x\;\le\; -\abs{x}, \quad \text{outside a compact set}\, \sK_1\,.$$
Taking $\Lyap(x)=\exp(\abs{x})$ for $\abs{x}\ge 1$, we have 
$$\Lg\Lyap\;=\; \Big(\frac{d-1}{2\abs{x}}+\frac{1}{2} +b(x, u)\cdot \frac{x}{\abs{x}}\Big)\Lyap\;\le\; \Big(\frac{d-1}{2\abs{x}}-\frac{1}{2}\Big)\Lyap, \quad \text{for}\; \abs{x}\ge 1\,.$$
\end{example}

Note that $\Lyap$ in Example~\ref{Eg1} is inf-compact. Below we produce an example where $\Lyap$ is not inf-compact.
\begin{example}
Let $d=1$, $\upsigma=\mathrm{Id}$ and $b(x, u)=-\mathrm{sgn}(x) \abs{x}^2 + u$ where $u\in[0, 1]$. Take $\Lyap(x) =\frac{\abs{x}}{1+\abs{x}}$
for $\abs{x}\geq 1$ and extend it in whole of $\RR$ as a smooth function so that $\inf_{\Rd}\Lyap>0$. Then for any $\gamma\in(0,1)$ we can find
a compact set $\sK$ and a constant $\beta$ such that
$$\sup_{u\in[0, 1]}\, \Lg^u\Lyap\; \leq \; \beta\, \Ind_{\sK} - \gamma\, \Lyap.$$
\end{example}

\begin{condition}\label{C2.2}
There exists positive functions $\Lyap\in \Cc^2(\Rd),\, \inf_{\Rd}\Lyap>0,$ and $\ell\in\Cc(\Rd), \ell$ inf-compact, such that
\begin{equation}\label{lyapunov1}
\sup_{u\in\Act_1\times\Act_2}(\Lg^u \Lyap) \;\le\; \beta \Ind_{\sK} - \ell\, \Lyap,
\end{equation}
for some constant $\beta$ and a compact set $\sK$. Moreover, the cost function $c$ belongs to $\sC_\ell$ where
$$\sC_\ell=\Bigl\{c:\Rd\times\Act_1\times\Act_2\to\RR_+\;\text{where}\; \limsup_{\abs{x}\to\infty}\tfrac
{\max_{u\in\Act_1\times\Act_2} c(x, u)}{\ell(x)}\;\le\; \theta\,\, \mbox{for some}\,\, \theta \in(0, 1) \Bigr\}.$$
\end{condition}
$\ell$ being inf-compact it is easy to see that for $c\in\sC_\ell$, $\bigl(\ell-\max_{u\in\Act_1\times\Act_2} c(x, u)\bigr)$ is inf-compact. Let us remark that if $a, b$ are bounded functions one can not expect \eqref{lyapunov1} to hold.
See for instance the Remark~3.4 in \cite{ABS}.

\begin{example}
Let $\upsigma$ be bounded and $b:\Rd\times\Act\to\Rd$ be such that 
$$\langle b(x, u)-b(0, u), x\rangle\; \;\le\; \;-\kappa\, \abs{x}^\alpha, \quad \text{for some}\; \alpha\in(1,2], \quad (x, u)\in\Rd\times\Act\,.$$
Then, $\Lyap(x)=\exp(\abs{x}^{\theta\alpha})$ for $\abs{x}\ge 1$,  satisfies \eqref{lyapunov1}
for sufficiently small $\theta>0$, and $\ell\sim\abs{x}^{2\alpha-1}$. Note that $\alpha=2, \, \upsigma=\mathrm{Id}$ is considered in \cite{Fleming-95} in the context of minimization problem with risk-sensitive criterion.
\end{example}

Let us now state the main result of this article

\begin{theorem}\label{T2.1}
Suppose that either Condition~\ref{C2.1} or Condition~\ref{C2.2} holds.
Then there exists an eigenpair $(V, \Lambda)\in\Cc^2(\Rd)\times\RR, \, V>0,$ that satisfies 
\begin{equation}\label{ET2.1A}
\max_{\tu_2\in\pAct_2}\min_{\tu_1\in\pAct_1}\Big(\frac{1}{2}a^{ij}\partial_{ij} V + b(x, \tu_1, \tu_2)\cdot\grad V
+ c(x, \tu_1, \tu_2)V \Big)\;=\;\Lambda\, V, \quad V(0)=1.
\end{equation}
Moreover, we have the following
\begin{enumerate}
\item[(i)] $\Lambda$ in \eqref{ET2.1A} is the value of game i.e., $\bar\Lambda=\underline{\Lambda}=\Lambda$.
\item[(ii)] If $v^*_2$ is an outer maximizing selector of 
\begin{equation}\label{ET2.1B}
\max_{\tu_2\in\pAct_2}\min_{\tu_1\in\pAct_1}\Big(\frac{1}{2}a^{ij}\partial_{ij} V + b(x, \tu_1, \tu_2)\cdot\grad V
+ c(x, \tu_1, \tu_2)V \Big)=\Lambda\, V,
\end{equation}
and $v^*_1$ is an outer minimizing selector of
\begin{equation}\label{ET2.1C}
\min_{\tu_1\in\pAct_1}\max_{\tu_2\in\pAct_2}\Big(\frac{1}{2}a^{ij}\partial_{ij} V + b(x, \tu_1, \tu_2)\cdot\grad V
+ c(x, \tu_1, \tu_2)V \Big)=\Lambda\, V,
\end{equation}
then the pair $(v^*_1, v^*_2)\in\Usm_1\times\Usm_2$ is a saddle point strategy i.e., 
$$\sE(v^*_1, U^2)\leq  \sE(v^*_1, v^*_2)=\Lambda \leq \sE(U^1, v^*_2),\quad \forall \, U^1\in\Uadm_1, \; \text{and}\;\; U^2\in\Uadm_2.$$
\item[(iii)] The eigenfunction $V$ is unique in the class $\Cc^2(\Rd)$, provided $V(0)=1$. 
\item[(iv)] If $(\hat{v}_1, \hat{v}_2)\in\Usm_1\times\Usm_2$ is a saddle point strategy in the above sense then $\hat{v}_1$ is an outer minimizing 
selector of \eqref{ET2.1C} and $\hat{v}_2$ is an outer maximizing selector of \eqref{ET2.1B}.
\end{enumerate}
\end{theorem}
For existence of measurable selector we refer the readers to \cite[Chapter~18]{Ali-Bor}.  Such existence 
of measurable selector will be used in several places in this article. Also note that by \cite[Theorem~3]{Fan}, the left hand sides of \eqref{ET2.1B} and \eqref{ET2.1C} are equal.
The proof of Theorem~\ref{T2.1} is established in Section~\ref{S-proofs} below.

Before we conclude this section let us recall the definition of principal eigenvalue from \cite{Berestycki-15}. Let
$$L\varphi = \tfrac{1}{2}\, \hat{a}^{ij}(x)\,\partial_{ij} \varphi (x)
+ \hat{b}^{i}(x)\, \partial_{i} \varphi (x)\, + \hat{c}(x)\varphi,$$
where $\hat{b}, \hat{c}$ are locally finite, Borel measurable functions, $\hat{a}$ is continuous and satisfies 
(A3). Then the principal eigenvalue of $L$ is defined to be
\begin{equation}\label{P-eigen}
\lamstr=\inf \{\lambda\; \colon\; \exists\ \varphi\in\Sobl^{2, d}(\Rd), \, \varphi>0,\;
\text{and}\; L\varphi-\lambda\varphi\leq 0,\; \text{a.e. in}\; \Rd\}.
\end{equation}
The following result is proved in \cite[Theorem~3.2 and 3.3]{ABS}.
\begin{lemma}\label{L2.1}
Suppose that either Conditon~\ref{C2.1} or Condition~\ref{C2.2} holds. Then for any $v=(v_1, v_2)\in\Usm_1\times\Usm_2$,
the principal eigenvalue of the operator 
$$\Lg_{v}\varphi=\tfrac{1}{2}\, a^{ij}(x)\,\partial_{ij} \varphi (x)
+ b^{i}(x, v(x))\, \partial_{i} \varphi (x)\, + c(x, v(x))\varphi,$$
is $\sE_x(v_1, v_2)$ for all $x\in\Rd$.
\end{lemma}

\section{\bf A maximization problem}\label{S-max}
In this section we study a maximization problem which will play a key role in our analysis. Let $r:\Rd\times\Act_2\to\RR_+$ and $b:\Rd\times\Act_2\to\RR^d$ be 
 Borel measurable functions that are locally finite and continuous in $u_2$ for every fixed $x$. Moreover, $b$ and $\upsigma$ satisfies (A2)-(A3).
We also assume that $\upsigma$ satisfies (A1). As earlier we consider
the relaxed action space $\tilde{\Act}_2$  and extend $b, r$ over $\tilde{\Act}_2$. Then the main result of this section
is the following
\begin{theorem}\label{T3.1}
Suppose that one of the following holds.
\begin{enumerate}
\item[(H1)] \eqref{lyapunov} holds and $\norm{r}_\infty<\;\gamma$.
\item[(H2)] \eqref{lyapunov1} holds with some continuous, inf-compact $\ell$ and 
\begin{equation}\label{cost}
\limsup_{\abs{x}\to\infty}\, \tfrac{\max_{u_2\in\Act_2} r(x, u_2)}{\ell(x)}\;\leq\; \theta\in(0, 1)\,.
\end{equation}
\end{enumerate}
Then there exists $(V, \lammax)\in\Sobl^{2, p}(\Rd)\times\RR, \, p\in(1, \infty),$ with $V>0$ and
\begin{equation}\label{ET3.1}
\max_{\tu_2\in\pAct_2}\Bigl(\tfrac{1}{2}\, a^{ij}(x)\,\partial_{ij} V (x)
+ b^{i}(x, \tu_2)\, \partial_{i} V (x)\, + r(x, \tu_2) V\Bigr)=\lammax\, V(x)\quad \text{a.e. in}\; \Rd.
\end{equation}
Moreover, we have the following:
\begin{enumerate}
\item[(i)] For all $x\in\Rd$, we have
$$\lammax\;=\; \sup_{U^2\in\Uadm_2}\, \sE_x(r, U^2)\;=\; \sup_{U^2\in\Uadm_2}\,
\limsup_{T\to\infty}\, \frac{1}{T}\, \log \Exp_x\left[\E^{\int_0^T r(X_s, U^2_s)\, \D{s}}\right].$$
\item[(ii)] Any measurable selector of \eqref{ET3.1} is an optimal stationary Markov control, i.e., 
if $v^*_2$ is a measurable maximizer of \eqref{ET3.1} then we have $\lammax=\sE_x(r, v^*_2)$.
\item[(iii)] The solution of \eqref{ET3.1} is unique in the class $\Sobl^{2, p}(\Rd), \, p\in(1, \infty)$, provided $V(0)=1$.
\item[(iv)] Any optimal stationary Markov control is a measurable selector of \eqref{ET3.1}.
\end{enumerate}
\end{theorem}

A similar minimization problem with long-run average of risk sensitive cost
  has been studied in \cite[Theorem~4.1 and 4.2]{ABS}. 
As we mentioned earlier it is not obvious that one can change a maximization problem to a minimization problem of same type.
Therefore proof Theorem~\ref{T3.1} does not follow from \cite{ABS}. We think Theorem~\ref{T3.1} would be of independent 
interest. The proof of Theorem~\ref{T3.1} is divided in several lemmas.

In what follows $B_n$ denotes the open ball of radius
$n$ around $0$. Following result follows from \cite[Theorem~1.1 and Remark~3]{Quaas-08a}
\begin{lemma}\label{L3.1}
There exists a (unique) positive $\varphi_n\in\Sobl^{2, p}(B_n)\cap\Cc(\bar{B}_n),\, p\in(1, \infty),$ and $\alpha_n\in\RR$
such that $\varphi_n(0)=1$ and
\begin{equation}\label{EL3.1}
\max_{\tu\in\pAct_2}\Bigl(\Lg^{\tu_2}\varphi_n + r(x, \tu_2)\varphi_n\Bigr)=\alpha_n\, \varphi_n(x),\quad \text{a.e. in}\; B_n,
\quad \varphi_n=0,\; \; \text{on}\; \partial B_n.
\end{equation}
Moreover, $\alpha_n<\alpha_{n+1}$ for all $n\in\NN$.
\end{lemma}

Let $v^*_{2, n}$ be a measurable selector of \eqref{EL3.1}. Then
\begin{equation}\label{E3.3}
\tfrac{1}{2}\, a^{ij}(x)\,\partial_{ij} \varphi_n (x)
+ b^{i}(x, v^*_{2, n})\, \partial_{i} \varphi_n (x)\, + r(x, v^*_{2, n}) \varphi_n =\;\alpha_n\, \varphi_n, \quad \text{a.e. in}\; B_n.
\end{equation}
Since $\varphi_n(0)=1$, by Harnack's inequality \cite[Corollary~9.25]{GilTru} and \eqref{E3.3} we have for any compact set $K\subset B_n$,
\begin{equation}\label{E3.4}
\sup_{K}\, \varphi_n\; \leq\; C(K),
\end{equation}
for some constant $C(K)$, not depending on $n$ but depends on $K$ and the constants in (A1)-(A3). Therefore using
standard Sobolev estimate \cite[Theorem~9.11]{GilTru} in \eqref{E3.3} we obtain 
$\{\varphi_n\,:\, n\geq 1\}$ uniformly bounded in $\Sob^{2, p}(K), p>d$. By a
standard diagonalization argument we can extract a subsequence of $\{\varphi_n\,:\, n\geq 1\}$ that converges to some 
$V\in\Sobl^{2, p}(\Rd),\, p>1$, strongly in $\Cc^{1, \alpha}_{\mathrm{loc}}(\Rd),\, \alpha\in(0, 1),$ and weakly in $\Sobl^{2, p}(\Rd),\, p>1$. Therefore we have the following lemma

\begin{lemma}\label{L3.2}
Suppose that either (H1) or (H2) of Theorem~\ref{T3.1} holds. Then
the sequence $\{\alpha_n\, :\, n\geq 1\}$ is bounded above. Moreover, 
if $(V, \lammax)$ is any sub-sequential limit of $(\varphi_n, \alpha_n)$, then we have
\begin{equation}\label{EL3.2A}
\max_{\tu_2\in\pAct_2}\Bigl(\Lg^{\tu_2} V(x) + r(x, \tu_2)V(x) \Bigr)\;=\; \lammax\, V(x), \quad \text{a.e. in}\; \Rd, \; \text{and}\; V>0.
\end{equation}
\end{lemma}

\begin{proof}
Let us first show that $\{\alpha_n\}$ is bounded above. Let $v^*_{2, n}$ be a measurable selector of \eqref{EL3.1}.
Extend the Markov control in $\Rd$ be settings $v^*_{2, n}(x)=u_2$ for all $x\in B^c_n$ where 
$u_2\in\pAct_2$ is fixed. Let $\uptau_n=\uptau(B_n)$. Then appying It\^{o}-Krylov formula \cite[p. 122]{Krylov} to \eqref{E3.3}
we have, for $x\in B_1$
\begin{align*}
\varphi_n(x) &= \Exp^{v^*_{2,n}}_x\left[\E^{\int_0^{T}[r(X_s, v^*_{2, n}(X_s)) -\alpha_n]\, \D{s}} 
\varphi_n(X_T)\Ind_{\{T<\uptau_n\}}\right]
\\
&\le\; \norm{\varphi_n}_\infty \,
\Exp^{v^*_{2,n}}_x\left[\E^{\int_0^{T}[r(X_s, v^*_{2, n}(X_s)) -\alpha_n]\, \D{s}} \right].
\end{align*}
Therefore taking logarithm on both sides, diving by $T$ and letting $T\to\infty$ we get
\begin{equation}\label{EL3.2B}
\alpha_n\leq \sE_x(r, v^*_{2, n}).
\end{equation}
Under condition~(H1) we have $r$ bounded and thus 
$$\sup_{ U^2\in\Uadm_2}\sE_x(r, U^2)\leq \norm{r}_\infty<\infty.$$
Suppose condition~(H2) holds. It is easy to see from \eqref{lyapunov1} that
$$\sup_{U^2\in\Uadm_2}\sE_x(\ell, U^2)\leq \tfrac{\beta}{\min_{\sK}\Lyap}.$$
By \eqref{cost},   $\max_{\tu_2\in\pAct_2}r(\cdot, \tu_2)\leq \kappa + \ell(\cdot)$ for  some positive $\kappa$.
Hence
$$\sup_{U^2\in\Uadm_2}\sE_x(r, U^2)\leq \kappa + \tfrac{\beta}{\min_{\sK}\Lyap}.$$
Combining these with \eqref{EL3.2B} we have
\begin{equation}\label{EL3.2C}
\alpha_n\leq \kappa_1, \quad \text{for all}\; n\geq 1,\; \text{and for some constant }\; \kappa_1.
\end{equation}
Hence first part of the lemma follows from \eqref{EL3.2C}.

Since $\{\alpha_n\,:\, n\geq 1\}$ is bounded above, $\lammax=\lim_{n\to\infty}\alpha_n$ exists. Let $V$ be a sub-sequential limit
of $\varphi_n$ as obtained above. Since $\varphi_n\to V$ in $\Cc^{1, \alpha}_{\mathrm{loc}}(\Rd)$ we have, for
any compact set $K$,
\begin{equation}\label{EL3.2D}
\sup_{x\in K}\Bigl|\max_{\tu_2\in\pAct_2}(b(x, \tu_2)\cdot \grad\varphi_n(x) + r(x, \tu_2)\varphi_n(x))
- \max_{\tu_2\in\pAct_2}(b(x, \tu_2)\cdot \grad V(x) + r(x, \tu_2) V(x)) \Bigr|
\end{equation}
tending to $0$, as $n\to\infty$. Now let $\chi$ be a smooth function with compact support. Then using the observation in \eqref{EL3.2D}
and using \eqref{EL3.1} we obtain, as $n\to\infty$, 
\begin{align}\label{EL3.2E}
 &\tfrac{1}{2}\int_{\Rd} a^{ij}(x) \partial_{ij} V(x) \chi(x) \, \D{x} 
 + \int_{\Rd} \max_{\tu_2\in\pAct_2}(b(x, \tu_2)\cdot \grad V(x) + r(x, \tu_2) V(x))\chi(x) \, \D{x}\nonumber
 \\
&\qquad\qquad =\lammax\int_{\Rd} V(x)\chi(x)\, \D{x}.
\end{align}
Since $\chi$ is arbitrary and $V\in\Sobl^{2, p}(\Rd), p>1,$  \eqref{EL3.2E} implies that 
$$\max_{\tu_2\in\pAct_2}\Bigl(\Lg^{\tu_2} V(x) + r(x, \tu_2)V(x) \Bigr)\;=\; \lammax\, V(x), \quad \text{a.e. in}\; \Rd.$$
It is easy to see that $V\geq 0$ and $V(0)=1$. Therefore by an application of Harnack's inequality in the above equation
we get $V>0$ in $\Rd$. Hence the proof.
\end{proof}

Next we show that $V$ has a stochastic representation.
\begin{lemma}\label{L3.3}
Suppose the assumptions of Theorem~\ref{T3.1} hold. Then for any sub-sequential limit $V$, as obtained in Lemma~\ref{L3.2},
there exists a compact set $\sB$ such that for any measurable selector $v^*_2$ of
\eqref{EL3.2A} and any compact ball $\sB_1\supset \sB$ we have
\begin{equation}\label{EL3.3A}
V(x)\;=\; \Exp^{v^*_2}_x\left[\E^{\int_0^{\uuptau_1}[r(X_s,  v^*_2(X_s))-\lammax]\, \D{s}} V(X_{\uuptau_1})
\right], \quad \text{for}\quad x\in\sB^c_1,
\end{equation}
where $\uuptau_1=\uptau(\sB^c_1)$.
\end{lemma}

\begin{proof}
Let us first argue that $\lammax$ is non-negative. If not, let us assume that $\lammax<0$. Let $v^*_2$ be a measurable
selector of \eqref{EL3.2A}. Existence of such a measurable selector is assured by \cite[Theorem~18.13]{Ali-Bor}.
Thus we have 
\begin{equation}\label{EL3.3C}
\Lg_{v^*_2} V(x) + r(x, v^*_2(x)) V(x)\; =\; \lammax V(x), \quad \text{a.e. in}\; \Rd.
\end{equation}
Applying It\^{o}-Krylov formula \cite[p. 122]{Krylov} to \eqref{EL3.3C} we obtain for any $T>0$ and for any $n$, that
\begin{equation}\label{EL3.3D}
V(x)\;=\; \Exp^{v^*_2}_x\left[\E^{\int_0^{\uptau(B^c_1)\wedge T\wedge \uptau_n }[r(X_s,  v^*_2(X_s))-\lammax]\, \D{s}} 
V(X_{\uptau(B^c_1)\wedge T\wedge \uptau_n})
\right], \quad \text{for}\quad x\in B^c_1.
\end{equation}
Letting $T\to\infty$, and $n\to\infty$ in \eqref{EL3.3D} and applying Fatou's lemma twice we get
$$V(x)\;\geq\; \Exp^{v^*_2}_x\left[\E^{\int_0^{\uptau(B^c_1)}[r(X_s,  v^*_2(X_s))-\lammax]\, \D{s}} 
V(X_{\uptau(B^c_1)}) \right]\;\geq\; \min_{\bar{B}_1}\, V, $$
where we have used the fact $(r -\lammax)\geq 0$ which follows from the fact that $r\geq 0$ and $\lammax<0$. Therefore we have $\min_{\Rd}\, V>0$. Again
a further application of It\^{o}-Krylov formula \cite[p. 122]{Krylov} to \eqref{EL3.3C}, followed by Fatou's lemma, gives us 
$$V(x) \;\geq\; \Exp^{v^*_2}_x\left[\E^{\int_0^{T}[r(X_s,  v^*_2(X_s))-\lammax]\, \D{s}} 
V(X_{T}) \right]\;\geq\; \min_{\Rd}V\,\Exp^{v^*_2}_x\left[\E^{\int_0^{T}[r(X_s,  v^*_2(X_s))-\lammax]\, \D{s}}  \right].$$
Now taking logarithm on both sides, diving by $T$ and letting $T\to \infty$ we obtain
$$\lammax\;\geq\;\sE_x(r, v^*_2)\;\geq\; 0.$$
But this is contradicting 
the fact that $\lammax<0$. Hence we have $\lammax\geq 0$.

First we note that there exists a compact set $\sB$ and $\hat{\theta}\in(0, 1)$ such that under (H1)
\begin{equation}\label{EL3.3E}
\Bigl(\max_{u\in\Act_2} r(x, u)-\lammax\Bigr)\;< \; \hat\theta\, \gamma, \quad \text{for all}\; x\in\sB^c,
\end{equation}
and under condition (H2),
\begin{equation}\label{EL3.3F}
\Bigl(\max_{u\in\Act_2} r(x, u)-\lammax\Bigr)\;< \; \hat\theta\, \ell(x), \quad \text{for all}\; x\in\sB^c.
\end{equation}
\eqref{EL3.3E} follows from the fact $0\leq \lammax\leq \norm{r}_\infty<\gamma$ whereas \eqref{EL3.3F}
follows from \eqref{cost}.  
Since $\alpha_n\to\lammax$ it is easy to see that we can find $\hat\theta\in(0,1)$ such that \eqref{EL3.3E} (and \eqref{EL3.3F})
holds true when $\lammax$ is replaced by $\alpha_n$ for all large $n$. Let  $\uuptau$ denote the hitting time to $\sB$.
Without loss of generality we may assume 
$\sB\supset \sK$ where $\sK$ is same as in \eqref{lyapunov} and \eqref{lyapunov1}. We give the proof
using \eqref{EL3.3E}, and the proof using \eqref{EL3.3F} follows by repeating the same argument.
$\varphi_n$ being positive we note that $\alpha_n$ is the principal eigenvalue of the elliptic operator on the LHS of \eqref{E3.3}.
Thus using \eqref{E3.3} and stochastic representation from \cite[Lemma~2.10(i)]{ari-anup} we have, for $x\in \sB^c$ 
and $\sB\subset B_n$, for all $n$ sufficiently large
\begin{align}\label{EL3.3FA}
\varphi_n(x) &\;=\;\Exp^{v^*_{2, n}}_x\left[\E^{\int_0^{\uuptau}[r(X_s, v^*_{2, n}(X_s))-\alpha_n]\, \D{s}} 
\varphi_n(X_{\uuptau_1})
\Ind_{\{\uuptau<\uptau_n\}}\right]\nonumber
\\
&\leq\; \sup_{\sB}\varphi_n\, \Exp^{v^*_{2, n}}_x\left[\E^{\hat{\theta}\gamma{\uuptau}} \right]\nonumber
\\
&\leq\; \sup_{\sB}\varphi_n\, \Bigl(\Exp^{v^*_{2, n}}_x\left[\E^{\gamma{\uuptau}} \right]\Bigr)^{\hat{\theta}}\nonumber
\\
&\leq\; \sup_{\sB}\varphi_n\,\frac{1}{\min_{\sB}\Lyap^{\hat{\theta}}} 
\Bigl(\Exp^{v^*_{2, n}}_x\left[\E^{\gamma{\uuptau}}\Lyap(X_{\uuptau}) \right]\Bigr)^{\hat{\theta}}\nonumber
\\
&\leq\; \kappa_2\, (\Lyap(x))^{\hat{\theta}},
\end{align}
for some constant $\kappa_2$, where the last inequality follows applying It\^{o}'s formula to \eqref{lyapunov}. Note that
$\kappa_2$ can be chosen independent of $n$, by \eqref{E3.4} thanks to the Harnack's inequality.  
Now we proceed to show the stochastic representation \eqref{EL3.3A}. Fix the compact set $\sB$ as chosen in 
\eqref{EL3.3E} and \eqref{EL3.3F}. Let $\sB_1$ be a compact ball containing $\sB$ and $\uuptau_1=\uptau(\sB_1^c)$.
As earlier we show \eqref{EL3.3A} under assertion (H1) and the proof under condition (H2) would be analogous.
Applying It\^{o}-Krylov
formula to \eqref{EL3.3C} we obtain, for any $T>0$,
\begin{equation}\label{EL3.3G}
V(x)\;=\; \Exp^{v^*_2}_x\left[\E^{\int_0^{\uuptau_1\wedge\uptau_n\wedge T}(r(X_s, v^*_2(X_s))-\lammax)\, \D{s}}\, 
V(X_{\uuptau_1\wedge\uptau_n\wedge T})\right], \quad x\in\sB^c_1\cap B_n, \; \sB_1\subset B_n\,,
\end{equation}
where $\uptau_n=\uptau(B_n)$ denotes the exit time from the ball $B_n$.
Since 
$$\Exp^{v^*_2}_x\left[\E^{\gamma\uuptau_1}\, \right]\;<\;\infty, \quad \text{for}\; x\in\sB^c_1,$$
by \eqref{lyapunov}, and $V$ is bounded in $\sB^c_1\cap B_n$, for every fixed $n$,  letting $T\to\infty$ in \eqref{EL3.3G}, we have
\begin{equation}\label{EL3.3H}
V(x)\;=\; \Exp^{v^*_2}_x\left[\E^{\int_0^{\uuptau_1\wedge\uptau_n}(r(X_s, v^*_2(X_s))-\lammax)\, \D{s}}\, 
V(X_{\uuptau\wedge\uptau_n})\right], \quad x\in\sB^c_1\cap B_n.
\end{equation}
On the other hand, since 
$$\Exp^{v^*_2}_x\left[\E^{\gamma(\uuptau_1\wedge\uptau_n)}\, \right]= 
\Exp^{v^*_2}_x\left[\E^{\gamma\,\uuptau_1}\, \Ind_{\{\uuptau_1<\uptau_n\}}\right] + \Exp^{v^*_2}_x\left[\E^{\gamma\,\uptau_n}\, \Ind_{\{\uuptau_1>\uptau_n\}}\right];$$
and by monotone convergence theorem the first two quantities converges to the same limit, we have 
\begin{equation}\label{EL3.3J}
\Exp^{v^*_2}_x\left[\E^{\gamma\,\uptau_n}\, \Ind_{\{\uuptau_1>\uptau_n\}}\right]\xrightarrow[n\to\infty]{} 0,\quad
\text{for}\; x\in \sB_1^c.
\end{equation}
We denote by $\Gamma(n, m)=\{x\in\partial B_n\; :\; V(x)\geq m\}$ for $m\geq 1$. Then using \eqref{EL3.3J}
\begin{align*}
&\Exp^{v^*_2}_x\left[\E^{\int_0^{\uptau_n}(r(X_s, v^*_2(X_s))-\lammax)\, \D{s}}\, 
V(X_{\uptau_n})\Ind_{\{\uptau_n<\uuptau_1\}}\right]
\\
&\quad \le\;
m\,\Exp^{v^*}_x\left[\E^{\gamma\, \uptau_n}\, \Ind_{\{\uptau_n<\uuptau_1\}}\right]
+ \Exp^{v^*}_x\left[\E^{\gamma\, \uptau_n}\, 
V(X_{\uptau_n})\Ind_{\{x\in\Gamma(n,m)\}}\Ind_{\{\uptau_n<\uuptau_1\}}\right]
\\
&\quad \le\;
m\,\Exp^{v^*}_x\left[\E^{\gamma\, \uptau_n}\, \Ind_{\{\uptau_n<\uuptau_1\}}\right]
+ \kappa_2\, \Exp^{v^*}_x\left[\E^{\gamma\, \uptau_n}\, 
(\Lyap(X_{\uptau_n}))^{\hat\theta}\Ind_{\{x\in\Gamma(n,m)\}}\Ind_{\{\uptau_n<\uuptau_1\}}\right]
\\
&\quad \le\;
m\,\Exp^{v^*}_x\left[\E^{\gamma\, \uptau_n}\, \Ind_{\{\uptau_n<\uuptau_1\}}\right]
+ \kappa_2 \Bigl[\frac{m}{\kappa_2}\Bigr]^{\frac{\hat\theta-1}{\hat\theta}}\, \Exp^{v^*}_x\left[\E^{\gamma\, \uptau_n}\, 
\Lyap(X_{\uptau_n})\Ind_{\{\uptau_n<\uuptau_1\}}\right]
\\
&\quad \le\;
m\,\Exp^{v^*}_x\left[\E^{\gamma\, \uptau_n}\, \Ind_{\{\uptau_n<\uuptau_1\}}\right]
+ \kappa_2 \Bigl[\frac{m}{\kappa_2}\Bigr]^{\frac{\hat\theta-1}{\hat\theta}}\, \Lyap(x)
\\
&\quad \longrightarrow 0\, ,
\end{align*}
by letting $n\to\infty$ first and then letting $m\to\infty$.
Therefore letting $n\to\infty$ in \eqref{EL3.3H} we have \eqref{EL3.3A}. Note that under (H2) in order to make analogous argument we use the fact that 
$$\Exp^{v^*_2}_x\left[\E^{\int_0^{\uuptau_1}\ell(X_t)dt}\, \right]\;<\;\infty, \quad \text{for}\; x\in\sB^c_1,$$
by \eqref{lyapunov1}.
\end{proof}

Now we are ready to prove Theorem~\ref{T3.1}. The idea of the proof is the following : First we perturb
 the cost to a {\it norm-like} cost and prove Theorem~\ref{T3.1} for the perturbed version. Then passing to the limit
 we show that the limiting eigenfunction also has stochastic representation \eqref{EL3.3A} and we use Lemma~\ref{L3.3} as a key ingredient in proving the result.
 
 Let us now introduce the perturbation of $r$, denoted by $r_m$. For $m\in\NN$, let $\zeta_m:\Rd\to[0, 1]$ 
 be a smooth function and satisfy 
 $\zeta_m(x)=1$ in $B_m$, $\zeta_m(x)=0$ in $B^c_{m+1}$.
For (H1), we choose $\delta$ small enough that so that $\norm{r}_\infty+\delta<\gamma$ and define 
$$r_m(x, \tu_2)= \zeta_m(x) r(x, u_2) + (1-\zeta_m(x)) (\norm{r}_\infty+\delta), \quad u_2\in\Act_2.$$
For (H2), we define
$$r_m(x, \tu_2)= r(x, u_2) + \frac{1}{m} \ell(x), \quad u_2\in\Act_2.$$
Choose $m$ large enough so that $r_m$ satisfies \eqref{cost}.
Note that $r_m$ is locally finite, continuous in
$\tu_2$ for every fixed $x$. Also Lemma~\ref{L3.1}, ~\ref{L3.2} and \ref{L3.3} holds with $r$ replaced by $r_m$. Thus we can
find an eigenpair $(V_m, \lambda_{2, m})\in\Sobl^{2, p}(\Rd)\times\RR$, $V_m>0,$ and
\begin{equation}\label{E3.16}
\max_{\tu_2\in\pAct_2}\Bigl(\Lg^{\tu_2} V_m(x) + r_m(x, \tu_2)V_m(x) \Bigr)\;=\; \lambda_{2, m}\, V_m(x), \quad \text{a.e. in}\; \Rd.
\end{equation}

\begin{lemma}\label{L3.4}
The solution $V_m$ of \eqref{E3.16} is bounded from below i.e., $\min_{\Rd} V_m>0$. Moreover, for all $x\in\Rd$,
\begin{equation}\label{EL3.4A}
\sup_{U^2\in\Uadm_2}\, \sE_x(r_m, U^2)\;=\; \lambda_{2, m}.
\end{equation}
\end{lemma}

\begin{proof}
Fix $m$. We claim that there exists a compact set $K$ such that $\min_{u_2\in\Act_2} r_m(x, u_2)-\lambda_{2,m}\geq 0$ for 
all $x\in K^c$. This obvious when $r_m$ is unbounded, as defined under condition (H2). Thus we prove the claim 
under condition (H1) where $r_m$ is bounded. Note that $r_m=\norm{r}_\infty + \delta$ for $x\in B^c_{m+1} $. Let
$v^*_{2, m}$ be a measurable selector of \eqref{E3.16}. Then by Lemma~\ref{L3.3}, $V_m$ has stochastic representation
with some compact ball $\sB_1$ and running cost $r_m$. From \cite[Corollary~2.3]{ABS} we know that only principal 
eigenvalue of the associated operator can have such representation. Therefore combing with Lemma~\ref{L2.1} we 
get $\lambda_{2, m}=\sE_x(r_m, v^*_{2, m})$ for all $x$. Again by strict monotonicity of principal eigenvalue 
\cite[Theorem~3.3]{ABS}
under condition (H1) we have $\lambda_{2, m}< \norm{r}_\infty + \delta$. Therefore choosing $K=\bar{B}_{m+1}$ we have 
our claim.

Without loss of generality we can choose $K$ large enough to contain $\sB$ where $\sB$ is given by Lemma~\ref{L3.3}.
Let $\uuptau_K$ be the hitting time to $K$. Hence by \eqref{EL3.3A}
$$V_m(x)\;=\; \Exp^{v^*_{2, m}}_x\left[\E^{\int_0^{\uuptau_K}[r_m(X_s,  v^*_2(X_s))-\lambda_{2, m}]\, \D{s}}
 V_m(X_{\uuptau_K})
\right]\;\geq\; \min_{K}\, V_m,\qquad\forall\, x\in K^c.$$
This proves that $\min_{\Rd} V_m=\min_{K} V_m >0$ (This is the key reason to perturb $r$ by $r_m$).
This proves first part of the lemma. To prove the second part
consider $U^2\in\Uadm_2$. Using \eqref{E3.16} and It\^{o}-Krylov  formula \cite[p. 122]{Krylov} we obtain
\begin{align*}
V_m(x) &\geq \; 
\Exp_x\left[ \E^{\int_0^{T}[r_m(X_s,  U^2(X_s))-\lambda_{2, m}]\, \D{s}}
 V_m(X_{T})\right]
 \\
 &\geq \; \min_{\Rd} V_m\; \Exp_x\left[ \E^{\int_0^{T}[r_m(X_s,  U^2(X_s))-\lambda_{2, m}]\, \D{s}}\right].
\end{align*}
Taking logarithm on both sides, dividing by $T$ and letting $T\to\infty$, we get
$$\sE_x(r_m, U^2)\;\leq\; \lambda_{2, m}.$$
$U^2$ in $\Uadm_2$ being arbitrary we have \eqref{EL3.4A}, using the above display and the fact $\lambda_{2, m}=\sE_x(r_m, v^*_{2, m})$.
\end{proof}

Let us fix $V_m(0)=1$. Recall that $v^*_{2, m}$ is a measurable selector of \eqref{E3.16} i.e.
\begin{equation}\label{E3.18}
\Lg_{v^*_{2, m}} V_m(x) + r_m(x, v^*_{2. m})V_m(x) \;=\; \lambda_{2, m}\, V_m(x), \quad \text{a.e. in}\; \Rd.
\end{equation}
Therefore applying Harnack's inequality and Sobolev estimate, as earlier, on \eqref{E3.18} it is easy to see that
the sequence $\{V_m \}$ is uniformly bounded in $\Sobl^{2, p}(\Rd), p>1$. Therefore we can extract a 
sub-sequence of $\{V_m \}$ that converges to some $V\in\Sobl^{2, p}(\Rd), p>1$, weakly in $\Sobl^{2, p}(\Rd),\, p>1,$
and strongly in $\Cc^{1, \alpha}_{\mathrm{loc}}(\Rd), \, \alpha\in(0, 1)$. On the other hand 
$\{\lambda_{2, m}\,:\, m\geq 1\}$ is also a bounded sequence. This follows from \eqref{EL3.4A} and by a similar reasoning as in Lemma~\ref{L3.2}. Let $(V, \lammax)$ be a sub-sequence
of $(V_m, \lambda_{2, m})$ as $m\to\infty$. Then following the arguments of Lemma~\ref{L3.2} we obtain
$$\max_{\tu_2\in\pAct_2}\Bigl(\Lg^{\tu_2} V(x) + r(x, \tu_2)V(x) \Bigr)\;=\; \lammax\, V(x), \quad \text{a.e. in}\; \Rd.$$
Again the arguments of Lemma~\ref{L3.3} shows that $V_m(x)\leq \kappa_2 (\Lyap(x))^{\hat\theta}, \, \hat\theta\in(0,1),$ uniformly in $m$ outside
a compact set (see \eqref{EL3.3FA}) and therefore, by an analogous calculation the limit $V$ also has stochastic representation. Summarizing we have the following lemma.

\begin{lemma}\label{L3.5}
Let the assumptions of Theorem~\ref{T3.1} hold.
There exists an eigenpair $(V, \lammax)\in\Sobl^{2, p}(\Rd)\times\RR, \, p\in(1, \infty),$ 
such that $V>0$, and
\begin{equation}\label{EL3.5A}
\max_{\tu_2\in\pAct_2}\Bigl(\tfrac{1}{2}\, a^{ij}(x)\,\partial_{ij} V (x)
+ b^{i}(x, \tu_2)\, \partial_{i} V (x)\, + r(x, \tu_2) V\Bigr)=\lammax\, V(x)\quad \text{a.e. in}\; \Rd.
\end{equation}
There exists a compact set $\sB$ such that for any measurable selector $v^*_2$ of
\eqref{EL3.5A} and any compact ball $\sB_1\supset \sB$ we have
\begin{equation}\label{EL3.5AA}
V(x)\;=\; \Exp^{v^*_2}_x\left[\E^{\int_0^{\uuptau_1}[r(X_s,  v^*_2(X_s))-\lammax]\, \D{s}} V(X_{\uuptau_1})
\right], \quad \text{for}\quad x\in\sB^c_1,
\end{equation}
where $\uuptau_1=\uptau(\sB^c_1)$. Moreover, for all $x\in\Rd$,
$$\lammax=\;\sup_{U^2\in\Uadm_2}\,\limsup_{T\to\infty}\, \frac{1}{T}\, \log \Exp_x\left[\E^{\int_0^T r(X_s, U^2_s)\, \D{s}}\right],$$
and any measurable selector of \eqref{EL3.5A} is an optimal stationary Markov control.
\end{lemma}

\begin{proof}
Select $(V, \lammax)$ as a sub-sequential limit of $(V_m, \lambda_{2, m})$. Then from the discussion 
preceding Lemma~\ref{L3.5} we see that
\eqref{EL3.5A} and \eqref{EL3.5AA} hold. Also from the construction of $r_m$ we have $r_m\geq r$ for all $(x, \tu_2)\in\Rd\times\pAct$. Therefore
by \eqref{EL3.4A} we have 
\begin{equation}\label{EL3.5B}
\sup_{U^2\in\Uadm_2}\, \sE_x(r, U^2)\;\leq \; \lammax, \quad \text{for all}\; x\in\Rd.
\end{equation}
Let $v^*_2$ be a measurable selector of \eqref{EL3.5A} i.e., 
\begin{equation}\label{EL3.5C}
\Bigl(\tfrac{1}{2}\, a^{ij}(x)\,\partial_{ij} V (x)
+ b^{i}(x, v^*_2)\, \partial_{i} V (x)\, + r(x, v^*_2) V\Bigr)=\lammax\, V(x).
\end{equation}
By \cite[Corollary~2.3]{ABS} only principal eigenvalue of the elliptic operator in
\eqref{EL3.5C} can have such stochastic representation \eqref{EL3.5AA}. Combining this with Lemma~\ref{L2.1} we note that
$\lammax=\sE_x(r, v^*_2)$ for all $x$. Thus by \eqref{EL3.5B},
$$\sup_{U^2\in\Uadm_2}\, \sE_x(r, U^2)\;= \; \lammax, \quad \text{for all}\; x\in\Rd,$$
and $v^*_2$ is an optimal stationary Markov control.
\end{proof}

Let us now prove Theorem~\ref{T3.1}
\begin{proof}[Proof of Theorem~\ref{T3.1}]
(i) and (ii) follows from Lemma~\ref{L3.5}. So we need to prove (iii) and (iv). To prove (iii) we assume that 
there exists positive $\widehat{V}\in\Sobl^{2, p}(\Rd), \, p\in (1, \infty),$ that satisfies 
\begin{equation}\label{ET3.1B}
\max_{\tu_2\in\pAct_2}\Bigl(\tfrac{1}{2}\, a^{ij}(x)\,\partial_{ij} \widehat{V} (x)
+ b^{i}(x, \tu_2)\, \partial_{i} \widehat{V} (x)\, + r(x, \tu_2) \widehat{V}\Bigr)\;=\;\lammax\, 
\widehat{V}(x)\quad \text{a.e. in}\; \Rd.
\end{equation}
Let $v^*_2$ be a measurable selector of \eqref{EL3.5A}. Using \eqref{ET3.1B}
\begin{equation}\label{ET3.1C}
\tfrac{1}{2}\, a^{ij}(x)\,\partial_{ij} \widehat{V} (x)
+ b^{i}(x, v^*_2)\, \partial_{i} \widehat{V} (x)\, + r(x, v^*_2) \widehat{V}\;\leq\;\lammax\, 
\widehat{V}(x)\quad \text{a.e. in}\; \Rd.
\end{equation}
Applying It\^{o}-Krylov formula to \eqref{ET3.1C}, with the same compact ball $\sB_1$ as in \eqref{EL3.5AA}, we obtain,
for $x\in\sB^c_1$,
\begin{equation}\label{ET3.1D}
\widehat{V}(x)\;\geq \; \Exp^{v^*_2}_x\left[\E^{\int_0^{\uuptau_1}[r(X_s,  v^*_2(X_s))-\lammax]\, \D{s}} \widehat{V}(X_{\uuptau_1})
\right].
\end{equation}
Using \eqref{EL3.5AA} and \eqref{ET3.1D} one has
$$\widehat{V}(x)-V(x)\geq \;
\Exp^{v^*_2}_x\left[\E^{\int_0^{\uuptau_1}[r(X_s,  v^*_2(X_s))-\lammax]\, \D{s}} \bigl(\widehat{V}(X_{\uuptau_1})
-V(X_{\uuptau_1})\bigr)\right].$$
Therefore if we multiply $V$ by a suitable positive constant so that $\widehat{V}-V$ is non-negative in $\sB_1$ and has its
minimum $0$ in $\sB_1$, the above display indicates that $\widehat{V}-V\geq 0$ in $\Rd$ with its minimum in $\sB_1$.
On the other hand by \eqref{EL3.5C} and \eqref{ET3.1C} one has
\begin{align*}
& \tfrac{1}{2}\, a^{ij}(x)\,\partial_{ij} (\widehat{V}-V)
+ b^{i}(x, v^*_2)\, \partial_{i} (\widehat{V}-V) \, - (r(x, v^*_2)-\lammax)^- (\widehat{V}-V)
\\
&\quad \leq\; - (r(x, v^*_2)-\lammax)^+\, (\widehat{V}-V)
\\
&\quad \leq \; 0.
\end{align*}
Hence applying strong maximum principle \cite[Theorem~9.6]{GilTru} one has $V=\widehat{V}$. This proves uniqueness.

Let us now prove (iv). Suppose $\tilde{v}$ is an optimal stationary Markov control, i.e., $\sE_x(r, \tilde{v})=\lammax$.
By \cite[Theorem~3.2 and 3.3]{ABS} there exists positive $\tilde{V}\in\Sobl^{2, p}(\Rd), \, p>1,$ such that
\begin{equation}\label{ET3.1E}
\tfrac{1}{2}\, a^{ij}(x)\,\partial_{ij} \tilde{V} (x)
+ b^{i}(x, \tilde{v})\, \partial_{i} \tilde{V} (x)\, + r(x, \tilde{v}) \tilde{V}\;=\;\lammax\tilde{V},
\end{equation}
and $\lammax=\sE_x(r, \tilde{v})$ for all $x\in\Rd$. Moreover $\tilde{V}$ will have stochastic representation with
respect to some compact ball $\sB$.
Now \eqref{EL3.5A} gives
\begin{equation}\label{ET3.1F}
\tfrac{1}{2}\, a^{ij}(x)\,\partial_{ij} {V} (x)
+ b^{i}(x, \tilde{v})\, \partial_{i} {V} (x)\, + r(x, \tilde{v}) {V}\;\leq\;\lammax\, {V}.
\end{equation}
Applying It\^{o}-Krylov formula we also have
\begin{align*}
\tilde{V}(x) &= \; \Exp^{\tilde{v}}_x\left[\E^{\int_0^{\uuptau}[r(X_s,  \tilde{v}(X_s))-\lammax]\, \D{s}} \tilde{V}(X_{\uuptau})
\right],
\\
{V}(x) &\geq \; \Exp^{\tilde{v}}_x\left[\E^{\int_0^{\uuptau}[r(X_s,  \tilde{v}(X_s))-\lammax]\, \D{s}} {V}(X_{\uuptau})
\right].
\end{align*}
Therefore with the help of these stochastic representations and \eqref{ET3.1E}-\eqref{ET3.1F} we can follow the same
argument as above (for uniqueness) to conclude that $V=\tilde{V}$. Thus the inequality in \eqref{ET3.1F} is in fact
an equality and $\tilde{v}$ is a measurable selector of \eqref{EL3.5A}. Hence the proof.
\end{proof}

\section{\bf Proof of the Theorem~\ref{T2.1}}\label{S-proofs}

In this section we prove Theorem~\ref{T2.1}. The proof is divided into several lemmas. All the results in this section are
valid under any one of the Conditions \ref{C2.1} and \ref{C2.2}. 
Recall that $B_n$ denotes the open ball of radius $n$ around $0$.
By $\mathbb{S}$ we denote the set of all real symmetric matrices.
 Let 
$F:\mathbb{S}\times\Rd\times \RR\times\RR^d\to \RR$ be defined
as follows 
$$F(M, p, v, x)= \max_{\tu_2\in\pAct_2}\min_{\tu_1\in\pAct_1}\Big(\frac{1}{2}a^{ij}(x)M^{ij} + b(x, \tu_1, \tu_2)\cdot p
+ c(x, \tu_1, \tu_2)v\Big).$$
Note that $F$ is linear in $M$ when other variables are kept fixed. Then by \cite[Theorem~2.2]{armstrong}, \cite{Ishii-Yoshimura} there exists 
an (unique) eigenpair $(\psi_n, \lambda_n)\in\Cc^{1, \alpha}_{\mathrm{loc}}(B_n)\cap\Cc(\bar{B}_n)\times\RR$ with $\psi_n>0$, such that
\begin{equation}\label{E4.1}
F(D^2\psi_n, D\psi_n, \psi_n, x)= \lambda_n\psi_n, \quad\text{in} \; B_n, \quad \text{and}\; \psi_n=0\; \text{on}\; \partial B_n.
\end{equation}
The above should be understood in the sense of viscosity solution. From the structure of $F$ it easy to see that $\psi_n$
is a viscosity solution of 
$$\frac{1}{2}a^{ij}(x)\partial_{ij} \Psi = f(x),$$
where 
$$f(x)\;\df\; -\max_{\tu_2\in\pAct_2}\min_{\tu_1\in\pAct_1}\Big( b(x, \tu_1, \tu_2)\cdot \grad\psi_n(x)
+ c(x, \tu_1, \tu_2)\psi_n(x)\Big) + \lambda_n\psi_n(x),$$
 is  locally H\"{o}lder continuous in $B_n$. Therefore using (A1), (A3) and \cite[Theorem~3]{Caffarelli} we get 
$\psi_n\in\Cc^{2, \alpha}_{\mathrm{loc}}(B_n)$ for some $\alpha>0$. We start with the following lemma.

\begin{lemma}\label{L4.1}
Suppose that either Condition~\ref{C2.1} or Condition~\ref{C2.2} holds.
For every $n\in\NN$ there exists an eigenpair $(\psi_n, \lambda_n)\in\Cc^{2}(B_n)\cap\Cc(\bar{B}_n)\times\RR$
with $\psi_n>0$ satisfying \eqref{E4.1}. Moreover, the set $\{\lambda_n\; \colon\; n\in\NN\}$ is bounded.
\end{lemma}

\begin{proof}
In view of the discussion above we only need to show that $\{\lambda_n\; \colon\; n\in\NN\}$ is a bounded set.
Note that by \cite[Theorem 3]{Fan}
\begin{align*}
&\max_{\tu_2\in\pAct_2}\min_{\tu_1\in\pAct_1}\Big(\frac{1}{2}a^{ij}(x)\partial_{ij}\psi_n(x) + b(x, \tu_1, \tu_2)\cdot \grad\psi_n(x)
+ c(x, \tu_1, \tu_2)\psi_n(x)\Big)
\\
&\quad =\min_{\tu_1\in\pAct_1}\max_{\tu_2\in\pAct_2}
\Big(\frac{1}{2}a^{ij}(x)\partial_{ij}\psi_n(x) + b(x, \tu_1, \tu_2)\cdot\grad\psi_n(x)
+ c(x, \tu_1, \tu_2)\psi_n(x)\Big).
\end{align*}
Let $v^n_2$ be a outer maximizing selector of the LHS above and $v^n_1$ be a outer minimizing selector of RHS above.
It is easy to check that 
\begin{equation}\label{EL4.1A}
\frac{1}{2}a^{ij}(x)\partial_{ij}\psi_n(x) + b(x, v^n_1, v^n_2)\cdot \grad\psi_n(x)
+ c(x, v^n_1, v^n_2)\psi_n(x)=\lambda_n\psi_n(x).
\end{equation}
Extend the Markov controls in $\Rd$ by setting $v^n_1(x)=u_1$ and $v^n_2(x)=u_2$ for all $x\in B^c_n$ where 
$(u_1, u_2)\in\pAct_1\times\pAct_2$ is fixed. Let $\uptau_n=\uptau(B_n)$. Then appying It\^{o}'s formula to \eqref{EL4.1A}
we have, for $x\in B_n$ and $v^n=(v^n_1, v^n_2)$, 
\begin{align*}
\psi_n(x) &= \Exp^v_x\left[\E^{\int_0^{T}[c(X_s, v^n_1(X_s), v^n_2(X_s)) -\lambda_n]\, \D{s}} \psi_n(X_T)\Ind_{\{T<\uptau_n\}}\right]
\\
&\le\; \norm{\psi_n}_\infty \,
\Exp^v_x\left[\E^{\int_0^{T}[c(X_s, v^n_1(X_s), v^n_2(X_s)) -\lambda_n]\, \D{s}} \right].
\end{align*}
Therefore taking logarithm on both sides, diving by $T$ and letting $T\to\infty$, we get
\begin{equation}\label{EL4.1B}
\lambda_n\leq \sE_x(v^n_1, v^n_2).
\end{equation}
Under Condition~\ref{C2.1} we have $c$ bounded and thus 
$$\sup_{(U^1, U^2)\in\Uadm_1\times\Uadm_2}\sE_x(U^1, U^2)\leq \norm{c}_\infty<\infty.$$
Suppose Condition~\ref{C2.2} holds. It is easy to see from \eqref{lyapunov1} that
$$\sup_{(U^1, U^2)\in\Uadm_1\times\Uadm_2}\sE_x(\ell, U^1, U^2)\leq \tfrac{\beta}{\min_{\sK}\Lyap}.$$
Since $c\in\sC_\ell$,   $\max_{u\in\Act_1\times\Act_2}c(\cdot, u)\leq \kappa + \ell(\cdot)$ for  some positive $\kappa$.
Hence
$$\sup_{(U^1, U^2)\in\Uadm_1\times\Uadm_2}\sE_x(c, U^1, U^2)\leq \kappa + \tfrac{\beta}{\min_{\sK}\Lyap}.$$
Combining these with \eqref{EL4.1B} we have
\begin{equation}\label{EL4.1C}
\lambda_n\leq \kappa_1, \quad \text{for all}\; n\geq 1,\; \text{and for some constant }\; \kappa_1.
\end{equation}
Thus it remains to show that the set $\{\lambda_n\; \colon\; n\in\NN\}$ is also bounded from below. We define
the elliptic operator $L^n$ as
$$L^n\varphi=\frac{1}{2}a^{ij}(x)\partial_{ij}\varphi(x) + b(x, v^n_1, v^n_2)\cdot \grad\varphi(x)
+ c(x, v^n_1, v^n_2)\varphi(x).$$
Then from \eqref{EL4.1A} we see that $(\psi_n, \lambda_n)$ is the principal eigenpair
of $L^n$ for the Dirichlet problem in $B_n$. If $\hat\lambda^n$ denote the principal eigenvalue of $L^n$ in $B_1$
then by monotonicity property of principal eigenvalues (with respect to domains ordered with respect to set inclusion, see for instance \cite{Berestycki-15})
 we know that $\lambda_n\geq \hat\lambda^n$. On the other
hand by \cite[Proposition~4.1]{Quaas-08a} there exists a constant $\kappa_0$, independent of $n$, such that
$\hat\lambda^n\geq \kappa_0$. Thus $\lambda_n\geq \kappa_0$ for all $n$. This completes the proof combining with \eqref{EL4.1C}.
\end{proof}

Set $\psi_n(0)=1$. Then by Harnack's inequality \cite[Corollary~9.25]{GilTru} and \eqref{EL4.1A} we have for any compact set $K\subset B_n$, $\{\psi_n\,:\, n\geq 1\}$ uniformly bounded in $\Sob^{2, p}(K), p>d$ (see \eqref{E3.4}). By a
standard diagonalization argument we can extract a subsequence of $\{\psi_n\,:\, n\geq 1\}$ that converges to some 
$V\in\Sobl^{2, p}(\Rd),\, p>1$, strongly in $\Cc^{1, \alpha}_{\mathrm{loc}}(\Rd),\, \alpha\in(0, 1),$ and weakly in $\Sobl^{2, p}(\Rd),\, p>1$.

\begin{lemma}\label{L4.2}
If $(V, \Lambda)$ is any sub-sequential limit of $(\psi_n, \lambda_n)$, as obtained above, then we have 
$V\in\Cc^2(\Rd)$ and 
\begin{equation}\label{EL4.2A}
\max_{\tu_2\in\pAct_2}\min_{\tu_1\in\pAct_1}\Big(\frac{1}{2}a^{ij}(x)\partial_{ij} V(x) + b(x, \tu_1, \tu_2)\cdot \grad V(x)
+ c(x, \tu_1, \tu_2) V(x)\Big)=\Lambda\, V(x), \quad V>0.
\end{equation}
Moreover, 
\begin{equation}\label{EL4.2B}
\Lambda\;\leq\; \underline{\Lambda} =\sup_{U^2\in\Uadm_2}\sJmin(x, U^2)= \sup_{U^2\in\Uadm_2}\inf_{U^1\in\Uadm_1}\sE_x(U^1, U^2).
\end{equation}
\end{lemma}

\begin{proof}
Since $\psi_n$ converges to $V$ along some subsequence, strongly in $\Cc^{1, \alpha}_{\mathrm{loc}}(\Rd),\, \alpha\in(0, 1),$ and weakly in $\Sobl^{2, p}(\Rd),\, p>1$, we can pass limit in \eqref{E4.1} to obtain \eqref{EL4.2A}. Regularity of $V$ can be
improved to $\Cc^2(\Rd)$ using ellipticity of $a$, (A1) and standard elliptic regularity estimates.
 Now fix $x\in\Rd$ and choose $n$ large
enough so that $x\in B_n$. Recall the outer maximizing selector $v^n_2$ from Lemma~\ref{L4.1}. Then for any $U^1\in\Uadm_1$ we have
$$ \Big(\frac{1}{2}a^{ij}(X_s)\partial_{ij} \psi_n(X_s) + b(X_s, U^1_s, v^n_2(X_s))\cdot \grad\psi_n(X_s)
+ c(X_s, U^1_s, v^n_2(X_s)) \psi_n(X_s)\Big)\;\geq\; \lambda_n\, \psi_n(X_s), $$
almost surely for $s<\uptau_n$.
Thus following a similar calculation as in Lemma~\ref{L4.1} (see \eqref{EL4.1B}) we obtain
$$\lambda_n\leq \sE_x(v^n_2, U^1).$$
Since $U^1$ has been chosen arbitrarily we get $\lambda_n\leq \inf_{U^1\in\Uadm_1}\sE_x(v^n_2, U^1)=
\sJmin(x, v^n_2)\leq \underline{\Lambda}$. Now let $n\to\infty$ to obtain \eqref{EL4.2B}.
\end{proof}

Following result shows that $\Lambda$, obtained in Lemma~\ref{L4.2}, is actually value of the game.

\begin{lemma}\label{L4.3}
Suppose that either Condition~\ref{C2.1} or Condition~\ref{C2.2} holds. Let $(V, \Lambda)\in\Cc^2(\Rd)\times\RR$ be 
the eigenpair obtained in Lemma~\ref{L4.2}. Then we have 
\begin{equation}\label{EL4.3A}
\bar{\Lambda} = \inf_{U^1\in\Uadm_1}\sJmax(x, U^1)= \inf_{U^1\in\Uadm_1}\sup_{U^2\in\Uadm_2}\sE_x(U^1, U^2)
\;\leq\; \Lambda.
\end{equation}
Thus we have $\bar\Lambda=\Lambda=\underline{\Lambda}$ for all $x\in\Rd$.
\end{lemma}

\begin{proof}
Fix an outer minimizer $v^*_1$ of \eqref{EL4.2A} and consider the maximization problem with $r(x, \tu_2)=c(x, v^*_1(x), \tu_2)$
$$\text{maximize}\; \sE_x(r, U^2), \quad \text{over}\; \Uadm_2.$$
From Theroem~\ref{T3.1} we know that there exists a stationary Markov control $w^*_2$ which is optimal for the 
above maximization problem. Moreover, if $\lammax$ is the maximum value we have $\lammax=\sE_x(c, v^*_1, w^*_2)$
for all $x$. On the other hand from \eqref{EL4.2A} we have
\begin{align*}
& \frac{1}{2}a^{ij}(x)\partial_{ij} V(x) + b(x, v^*_1, w^*_2)\cdot \grad V(x)
+ c(x, v^*_1, w^*_2) V(x)
\\
&\quad \leq\; \max_{\tu_2\in\pAct}\Bigl(\frac{1}{2}a^{ij}(x)\partial_{ij} V(x) + b(x, v^*_1, \tu_2)\cdot \grad V(x)
+ c(x, v^*_1, \tu_2) V(x)\Bigr)
\\
&\quad =\; \min_{\tu_1\in\pAct}\max_{\tu_2\in\pAct}\Bigl(\frac{1}{2}a^{ij}(x)\partial_{ij} V(x) + b(x, \tu_1, \tu_2)\cdot \grad V(x)
+ c(x, \tu_1, \tu_2) V(x)\Bigr)
\\
&\quad=\; \Lambda \, V(x)\, .
\end{align*}
Therefore by the definition of principal eigenvalue \eqref{P-eigen} we see that $\Lambda$   is bigger than the principal eigenvalue of 
$\Lg_{(v^*_1, w^*_2)} + c(x, v^*_1, w^*_2)$. Combining with Lemma~\ref{L2.1} we have $\sE_x(c, v^*_1, w^*_2)
\;\leq \Lambda$.
Thus
\begin{equation}\label{EL3.4B}
\Lambda\geq \sup_{U^2\in\Uadm_2}\, \sE_x(v^*_1, U^2)\;\geq\; \inf_{U^1\in\Uadm_1}\sJmax(x, U^1)=\bar{\Lambda}.
\end{equation}
This proves \eqref{EL4.3A}.
Since $\underline{\Lambda}\leq \bar{\Lambda}$ in general, it follows from Lemma~\ref{L4.2} that
$\bar\Lambda=\Lambda=\underline{\Lambda}$ for all $x\in\Rd$.
\end{proof}

Now onwards we fix an outer maximizing strategy $v^*_2$ and outer minimizing strategy $v^*_1$ of \eqref{EL4.2A}.
Let us now show that the value $\Lambda$ is achieved by applying the strategy $v^*=(v^*_1, v^*_2)$.

\begin{lemma}\label{L4.4}
Suppose that either Condition~\ref{C2.1} or Condition~\ref{C2.2} holds.
Let $(V, \Lambda)$ be an eigenpair obtained in Lemma~\ref{L4.2}. Then we have 
a compact ball $\sB$ such that
for any compact ball $\sB_1\supset\sB$ we have
\begin{equation*}
V(x) \;=\; \Exp^{v^*}_x\left[\E^{\int_0^{\uuptau_1}[c(X_s, v^*_1(X_s), v^*_2(X_s))-\Lambda]\, \D{s}} V(X_{\uuptau_1})
\right], \quad \text{for}\quad x\in\sB^c_1,
\end{equation*}
where $\uuptau_1=\uptau(\sB^c_1)$.Moreover, $\Lambda=\sE_x(v^*_1, v^*_2)$.
\end{lemma}

\begin{proof}
From \eqref{EL4.2A} we have
\begin{equation}\label{EL4.4A}
\frac{1}{2}a^{ij}(x)\partial_{ij} V + b(x, v^*_1, v^*_2)\cdot \grad V
+ c(x, v^*_1, v^*_2) V=\Lambda\, V.
\end{equation}
Then following the calculations of Lemma~\ref{L3.3}, we can find a compact set $\sB$ such that
for any compact ball $\sB_1\supset\sB$ we have
\begin{equation}\label{EL4.4B}
V(x) \;=\; \Exp^{v^*}_x\left[\E^{\int_0^{\uuptau_1}[c(X_s, v^*_1(X_s), v^*_2(X_s))-\Lambda]\, \D{s}} V(X_{\uuptau_1})
\right], \quad \text{for}\quad x\in\sB^c_1,
\end{equation}
where $\uuptau_1=\uptau(\sB^c_1)$.
But \eqref{EL4.4B} is known as the
stochastic representation of $V$, and therefore by \cite[Corollary~2.3]{ABS} $\Lambda$ is the principal eigenvalue of
the elliptic operator \eqref{EL4.4A} in $\Rd$. Hence, by Lemma~\ref{L2.1}, $\Lambda=\sE_x(v^*_1, v^*_2)$ for all $x\in\Rd$.
\end{proof}
Thus so far we have shown that the game has a value $\Lambda$, and the value is attained by the strategy $(v^*_1, v^*_2)$.
Let us now show that $(v^*_1, v^*_2)$ is in fact a saddle point strategy.

\begin{lemma}\label{L4.5}
Suppose that either Condition~\ref{C2.1} or Condition~\ref{C2.2} holds.
Let $(v^*_1, v^*_2)$ be same as we have chosen in Lemma~\ref{L4.4}. Then we have for all $x\in\Rd$,
\begin{equation}\label{EL4.5A}
\sE_x(v^*_1, U^2) \;\leq\;  \sE(v^*_1, v^*_2) \;\leq\; \sE_x(U^1, v^*_2),\quad \forall \; U^1\in\Uadm_1, \; \text{and}\; U^2\in\Uadm_2.
\end{equation}
In other words, $(v^*_1, v^*_2)$ is a saddle point strategy.
\end{lemma}

\begin{proof}
The first inequality in \eqref{EL4.5A} follows from \eqref{EL3.4B}.
 So we only prove the second inequality. Let us first
show that for any stationary Markov control $v:\Rd\to\pAct_1$ we have
\begin{equation}\label{EL4.5B}
\sE(v^*_1, v^*_2)\;\leq\; \sE(v, v^*_2).
\end{equation}
The proof of \eqref{EL4.5B} is based on the method of contradiction. 
Suppose that $\Lambda=\sE(v^*_1, v^*_2)> \sE(v, v^*_2)=\lambda_v$.
Then following the same argument
as in \cite[Theorem~4.1 and ~4.2]{ABS} (see also Section~\ref{S-max}) we can find a positive eigenfunction 
$\widehat{V}\in\Sobl^{2, p}(\Rd), \, p>1,$ such that
\begin{equation}\label{EL4.5C}
\Big(\frac{1}{2}a^{ij}\partial_{ij}\widehat{V} + b(x, v, v^*_2)\cdot\grad \widehat{V}
+ c(x, v, v^*_2)\widehat{V}\Big)=\lambda_v \widehat{V},
\end{equation}
and for a suitable compact set $\sB$ we have for $\sB_1\supset\sB$ hitting time $\uuptau_1$ 
\begin{equation}\label{EL4.5D}
\widehat{V}(x)= \Exp^{\hat{v}}_x\left[\E^{\int_0^{\uuptau_1}[c(X_s, \hat{v}(X_s))-\lambda_v]\, \D{s}} \widehat{V}(X_{\uuptau_1})
\right], \quad \text{for}\quad x\in\sB^c_1, \, \hat{v}=(v, v^*_2).
\end{equation}
This can be obtained by a similar argument as in Lemma~\ref{L3.3}.
On the other hand we have from \eqref{EL4.2A} that
\begin{equation}\label{EL4.5E}
\Big(\frac{1}{2}a^{ij}\partial_{ij}V + b(x, v, v^*_2)\cdot\grad V
+ c(x, v, v^*_2)V\Big)\geq \Lambda V.
\end{equation}
Using the stochastic representation in Lemma~\ref{L4.4} we find that for some constants $\kappa_2, \hat{\theta}\in(0,1)$, 
$V(x)\leq \kappa\, (\Lyap(x))^{\hat\theta}$.
Therefore following a similar calculation as in Lemma~\ref{L3.3} and using \eqref{EL4.5E} we obtain
\begin{equation}\label{EL4.5F}
V(x) \;\leq\; \Exp_x\left[\E^{\int_0^{\uuptau_1}[c(X_s, v, v^*_2)-\Lambda]\, \D{s}} V(X_{\uuptau_1})
\right], \quad \text{for}\quad x\in\sB^c_1,
\end{equation}
for some large compact ball $\sB_1$.
Since $\Lambda>\lambda_v$ we have from \eqref{EL4.5D} and \eqref{EL4.5F} that for $x\in\sB^c_1$,
\begin{align*}
\widehat{V}(x)-V(x) &\geq\; \Exp_x\left[\E^{\int_0^{\uuptau_1}[c(X_s, v, v^*_2)-\lambda_v]\, \D{s}} \widehat{V}(X_{\uuptau_1})
\right]-
\Exp_x\left[\E^{\int_0^{\uuptau_1}[c(X_s, v, v^*_2)-\Lambda]\, \D{s}} V(X_{\uuptau_1})
\right]
\\
&\geq\; \Exp_x\left[\E^{\int_0^{\uuptau_1}[c(X_s, v, v^*_2)-\Lambda]\, \D{s}} \bigl(\widehat{V}(X_{\uuptau_1})-V(X_{\uuptau_1})\bigr)\right].
\end{align*}
Hence we can multiply $V$ by a suitable positive constant so that $\widehat{V}-V\geq 0$ and attends its minimum $0$
in $\sB_1$. Moreover, from \eqref{EL4.5C} and \eqref{EL4.5E} we also have
\begin{align*}
&\Big(\frac{1}{2}a^{ij}\partial_{ij}(\widehat{V}-V) + b(x, v, v^*_2)\cdot\grad (\widehat{V}-V)
- (c(x, v, v^*_2)-\lambda_v)^- (\widehat{V}-V)\Big)
\\
& \;\leq\; -(c(x, v, v^*_2)-\lambda_v)^+ (\widehat{V}-V)
-(\Lambda-\lambda_v) V
\\
&\;\leq\; 0.
\end{align*}
Thus by strong maximum principle  \cite[Theorem~9.6]{GilTru} we should have $V=\widehat{V}$ which would lead to $\lambda_v\geq\Lambda$ from \eqref{EL4.5C} and \eqref{EL4.5E}. This is a contradiction. Hence we have \eqref{EL4.5B}.

Now to compete the proof of \eqref{EL4.5A} we consider the following minimization problem
$$\inf_{U^1\in\Uadm_1} \sE(U^1, v^*_2).$$
But this problem has a minimizer in the class of stationary Markov controls. This can be seen mimicking the arguments of 
\cite[Theorem~4.1 and ~4.2]{ABS} for measurable cost functions (see also Section~\ref{S-max} for a similar argument).
 Then the second inequality in \eqref{EL4.5A} follows from \eqref{EL4.5B}.
\end{proof}

Now we are ready to prove the uniqueness of the value function.

\begin{lemma}\label{L4.6}
Let $\Lambda$ be the value of the game, as obtained above. Then there exists a unique $V\in\Cc^2(\Rd)$ that satisfies \eqref{EL4.2A} with $V(0)=1$.
\end{lemma}

\begin{proof}
Suppose that $\tilde{V}\in\Cc^2(\Rd)$ is another solution to
\begin{align}\label{EL4.6A}
&\min_{\tu_1\in\pAct_1}\max_{\tu_2\in\pAct_2}\Big(\frac{1}{2}a^{ij}\partial_{ij}\tilde{V} + b(x, \tu_1, \tu_2)\cdot\grad \tilde{V}
+ c(x, v_1, v_2)\tilde{V}\Big)\nonumber
\\
&\quad =\max_{\tu_2\in\pAct_2}\min_{\tu_1\in\pAct_1}\Big(\frac{1}{2}a^{ij}\partial_{ij}\tilde{V} + b(x, \tu_1, \tu_2)\cdot\grad \tilde{V}
+ c(x, \tu_1, \tu_2)\tilde{V}\Big) =\Lambda \tilde{V}.
\end{align}
Let $\tilde{v}^*_1$ be an outer minimizing selector of \eqref{EL4.6A}. Recall $v^*_2$ from Lemma~\ref{L4.4}. 
Then we have from
\eqref{EL4.6A} that
\begin{align}\label{EL4.6B}
&\Big(\frac{1}{2}a^{ij}\partial_{ij}\tilde{V} + b(x, \tilde{v}^*_1, v^*_2)\cdot\grad \tilde{V}
+ c(x, \tilde{v}^*_1, v^*_2)\tilde{V}\Big)\nonumber
\\
&\quad \leq\; \max_{\tu_2\in\pAct_2}\Big(\frac{1}{2}a^{ij}\partial_{ij}\tilde{V} + b(x, \tilde{v}^*_1, \tu_2)\cdot\grad \tilde{V}
+ c(x, \tilde{v}^*_1, \tu_2)\tilde{V}\Big)=\Lambda\tilde{V}
\end{align}
Let $\tilde{v}^*=(\tilde{v}^*_1, v^*_2)$ and $c_{\tilde{v}^*}=c(x, \tilde{v}^*_1(x), v^*_2(x))$.
Applying It\^{o}'s formula to \eqref{EL4.6B} 
and then Fatou's lemma we have
for any closed ball $\sB$
\begin{equation}\label{EL4.6C}
\tilde{V}(x)\;\geq\; \Exp^{\tilde{v}^*}_x\left[\E^{\int_0^{\uuptau}[c_{\tilde{v}^*}(X_s)-\Lambda]\, \D{s}} \tilde{V}(X_{\uuptau})
\right], \quad \text{for}\quad x\in\sB^c.
\end{equation}
On the other hand from \eqref{EL4.2A} we get
\begin{align}\label{EL4.6D}
&\Big(\frac{1}{2}a^{ij}\partial_{ij}V + b(x, \tilde{v}^*_1, v^*_2)\cdot\grad V
+ c(x, \tilde{v}^*_1, v^*_2)V\Big)\nonumber
\\
&\quad \geq\; \min_{\tu_1\in\pAct_1}\Big(\frac{1}{2}a^{ij}\partial_{ij}V + b(x, \tu_1, v^*_2)\cdot\grad V
+ c(x, \tu_1, v^*_2) V\Big)\nonumber
\\
&\quad =\; \max_{\tu_2\in\pAct_2}\min_{\tu_1\in\pAct_1}\Big(\frac{1}{2}a^{ij}\partial_{ij}V + b(x, \tu_1, \tu_2)\cdot\grad V
+ c(x, \tu_1, \tu_2) V\Big)\nonumber
\\
&\quad=\; \Lambda V
\end{align}
Since $V\leq \kappa_2 (\Lyap)^{\hat\theta}$ by Lemma~\ref{L4.4},
applying It\^{o}'s formula as before (see Lemma~\ref{L3.3}) to \eqref{EL4.6D} we obtain, for a suitable closed ball $\sB$,
\begin{equation}\label{EL4.6E}
V(x) \;\leq\; \Exp^{\tilde{v}^*}_x\left[\E^{\int_0^{\uuptau}[c_{\tilde{v}^*}(X_s)-\Lambda]\, \D{s}} V(X_{\uuptau})
\right], \quad \text{for}\quad x\in\sB^c.
\end{equation}
Now use \eqref{EL4.6C} and \eqref{EL4.6E}, scale ${V}$ by multiplying a suitable positive constant,
 so that $\tilde{V}-V\geq 0$ and it attains its minimum $0$ in 
$\sB$. Again from \eqref{EL4.6B} and \eqref{EL4.6D} we have
\begin{align*}
&\Big(\frac{1}{2}a^{ij}\partial_{ij}(\tilde{V}-V) + b(x, \tilde{v}^*_1, v^*_2)\cdot\grad (\tilde{V}-V)
- (c_{\tilde{v}^*}(x)-\Lambda)^- (\tilde{V}-V)\Big)
\\
& \leq -(c_{\tilde{v}^*}(x)-\Lambda)^+ (\tilde{V}-V)
\\
&\leq 0.
\end{align*}
Applying strong maximum principle  \cite[Theorem~9.6]{GilTru} we have $V=\tilde{V}$. Hence the proof.
\end{proof}

Finally we are left with proving the verification result for saddle point strategy.

\begin{lemma}\label{L4.7}
Suppose that either Condition~\ref{C2.1} or Condition~\ref{C2.2} holds.
Consider a  stationary Markov control pair $(\hat{v}_1, \hat{v}_2)\in\Usm_1\times\Usm_2$ which is a saddle point strategy i.e.,
\begin{align*}
\sE(\hat{v}_1, \hat{v}_2) &\;\leq\; \sE(U^1, \hat{v}_2), \quad \text{for all}\; U^1\in\Uadm_1,
\\
\sE(\hat{v}_1, \hat{v}_2) &\;\geq\; \sE(\hat{v}_1, U^2), \quad \text{for all}\; U^2\in\Uadm_2.
\end{align*}
Then $\hat{v}_1$ is an outer minimizing selector of \eqref{ET2.1C}, and $\hat{v}_2$ is an outer maximizing selector of
\eqref{ET2.1B}.
\end{lemma}

\begin{proof}
Since 
\begin{align*}
 \Lambda=\inf_{U^1\in\Uadm_1} \sJmax(x, U^1)&\leq\; \sup_{U^2\in\Uadm_2}\sE_x(\hat{v}_1, U^2)
 \\
& \leq\; \sE_x(\hat{v}_1, \hat{v}_2)\leq\inf_{U^1\in\Uadm_1}\sE_x(U^1, \hat{v}_2)\leq
 \sup_{U^2\in\Uadm_2} \sJmin(x, U^2)=\Lambda,
\end{align*}
we have $\sE(\hat{v}_1, \hat{v}_2)=\Lambda$.
We give the proof under Condition~\ref{C2.1} and proof with the other condition is analogous.
So we assume that Condition~\ref{C2.1} holds.
We give the proof for the first situation and proof for the other situation would be 
analogous.
Therefore we 
show that $\hat{v}_2$ is an outer maximizing selector of \eqref{ET2.1B}. Fix $\hat{v}_2$ and consider the cost function
$\hat{c}(x, u_1)= c(x, u_1, \hat{v}_2(x))$. Then from \cite[Theorem~4.2]{ABS} we can find $\widehat{V}\in\Sobl^{2, p}(\Rd)$ that satisfies 
\begin{equation}\label{EL4.7A}
\min_{\tu_1\in\pAct_1}\Big(\frac{1}{2}a^{ij}\partial_{ij}\widehat{V} + b(x, \tu_1, \hat{v}_2)\cdot\grad \widehat{V}
+ \hat{c}(x, \tu_1) \widehat{V}\Big)=\Lambda\widehat{V},
\end{equation}
and for some positive constant $\kappa_2, \hat\theta\in(0,1)$, $\widehat{V}\leq\kappa_2\, (\Lyap)^{\hat\theta}$.
Now choosing the control to $v_1^*$ (which is same as in Lemma~\ref{L4.4})
we have from \eqref{EL4.7A}
\begin{equation}\label{EL4.7B}
\Big(\frac{1}{2}a^{ij}\partial_{ij}\widehat{V} + b(x, v^*_1, \hat{v}_2)\cdot\grad \widehat{V}
+ c(x, v^*_1, \hat{v}_2) \widehat{V}\Big)\geq \Lambda\widehat{V}.
\end{equation}
Applying It\^{o}'s formula to \eqref{EL4.7B} and following an argument similar to Lemma~\ref{L3.3}
we have
\begin{equation}\label{EL4.7C}
\widehat{V}(x)\leq \Exp_x\left[\E^{\int_0^{\uuptau}[c(X_s, v^*_1, u^*_2)-\Lambda]\, \D{s}} \widehat{V}(X_{\uuptau})
\right], \quad \text{for}\quad x\in\sB^c,
\end{equation}
for some closed ball $\sB$ and $\uuptau=\uptau(\sB^c)$.
On the other hand from \eqref{EL4.2A} we have
\begin{align}\label{EL4.7D}
\Lambda V&= \max_{v_2\in\pAct}\Big(\frac{1}{2}a^{ij}\partial_{ij}V + b(x, v^*_1, v_2)\cdot\grad V
+ c(x, v^*_1, v_2)V\Big)\nonumber
\\
&\geq \Big(\frac{1}{2}a^{ij}\partial_{ij}V + b(x, v^*_1, \hat{v}_2)\cdot\grad V
+ c(x, v^*_1, \hat{v}_2)V\Big),
\end{align}
and for the same closed ball $\sB$
\begin{equation}\label{EL4.7E}
V(x)\geq \Exp_x\left[\E^{\int_0^{\uuptau}[c(X_s, v^*_1, u^*_2)-\Lambda]\, \D{s}} V(X_{\uuptau})
\right], \quad \text{for}\quad x\in\sB^c.
\end{equation}
Now following the proof of Lemma~\ref{L4.6} and  using \eqref{EL4.7B}-\eqref{EL4.7E} we get $V=\widehat{V}$.
Therefore from \eqref{EL4.7A} and \eqref{ET2.1B} we see that $\hat{v}_2$ is an outer maximizing selector. Hence the proof.
\end{proof}

Finally we are ready to prove Theorem~\ref{T2.1}.

\begin{proof}[Proof of Theorem~\ref{T2.1}]
Existence of $(V, \Lambda)\in\Cc^2(\Rd)\times \RR$ which satisfies \eqref{ET2.1A}
 follows from Lemma~\ref{L4.2}
and Lemma~\ref{L4.4}. Lemma~\ref{L4.3} implies (i). (ii) follows from Lemma~\ref{L4.5} and (iii) follows from Lemma~\ref{L4.6}.
Verification part (iv) follows from Lemma~\ref{L4.7}.
\end{proof}

\subsection*{Acknowledgements.} We thank the anonymous referees for their careful reading and valuable comments.
This research of Anup Biswas was partly supported by an INSPIRE faculty fellowship and a DST-SERB grant EMR/2016/004810. Subhamay Saha acknowledges the hospitality of the Department of Mathematics in IISER-Pune while he was visiting at the early stages of this work.

\bibliographystyle{abbrv}
\bibliography{BS-risk}

\end{document}